\documentclass{amsart}

\usepackage[T1]{fontenc}
\usepackage{enumerate, amsmath, amsfonts, amssymb, amsthm, thmtools, mathrsfs, wasysym, graphics, graphicx, xcolor, url, hyperref, hypcap, shuffle, xargs, multicol, overpic, pdflscape, multirow, hvfloat, minibox, accents, array, multido, xifthen, a4wide, ae, aecompl, blkarray, pifont, mathtools, etoolbox, dsfont}
\usepackage{marginnote}
\hypersetup{colorlinks=true, citecolor=darkblue, linkcolor=darkblue}
\usepackage[all]{xy}
\usepackage[bottom]{footmisc}
\usepackage{tikz}
\usetikzlibrary{trees, decorations, decorations.markings, shapes, arrows, matrix, calc, fit, intersections, patterns, angles, cd}
\usepackage[external]{forest}
\graphicspath{{figures/}}
\makeatletter\def\input@path{{figures/}}\makeatother
\usepackage{caption}
\usepackage[noabbrev,capitalise]{cleveref}
\usepackage[export]{adjustbox}
\usepackage{ulem}

\newtheorem{theorem}{Theorem}
\newtheorem{corollary}[theorem]{Corollary}
\newtheorem{proposition}[theorem]{Proposition}
\newtheorem{lemma}[theorem]{Lemma}
\newtheorem{conjecture}[theorem]{Conjecture}
\newtheorem*{theorem*}{Theorem}
\newtheorem*{proposition*}{Proposition}
\newtheorem*{conjecture*}{Conjecture}

\theoremstyle{definition}
\newtheorem{definition}[theorem]{Definition}
\newtheorem{example}[theorem]{Example}
\newtheorem{remark}[theorem]{Remark}
\newtheorem{question}[theorem]{Question}

\newtheorem{defprop}[theorem]{Definition/Proposition}
\newtheorem{warning}[theorem]{Warning}
\crefname{notation}{Notation}{Notations}
\crefname{problem}{Problem}{Problems}
 
\newcommand{\R}{\mathbb{R}} 
\newcommand{\HH}{\mathbb{H}} 
\renewcommand{\b}[1]{{\boldsymbol{#1}}} 

\newcommand{\set}[2]{\left\{ #1 \;\middle|\; #2 \right\}} 
\newcommand{\bigset}[2]{\big\{ #1 \;\big|\; #2 \big\}} 
\newcommand{\setangle}[2]{\left\langle #1 \;\middle|\; #2 \right\rangle} 
\newcommand{\ssm}{\smallsetminus} 
\newcommand{\dotprod}[2]{\left\langle \, #1 \; \middle| \; #2 \, \right\rangle} 
\newcommand{\symdif}{\,\triangle\,} 
\newcommand{\one}{\b{1}} 
\newcommand{\eqdef}{\mbox{\,\raisebox{0.2ex}{\scriptsize\ensuremath{\mathrm:}}\ensuremath{=}\,}} 

\DeclareMathOperator{\inv}{inv} 
\DeclareMathOperator{\ninv}{ninv} 
\DeclareMathOperator{\Hom}{Hom} 

\newcommand{\ie}{\textit{i.e.}~} 
\newcommand{\eg}{\textit{e.g.}~} 
\newcommand{\aka}{\textit{a.k.a.}~} 
\definecolor{darkblue}{rgb}{0,0,0.7} 
\definecolor{green}{RGB}{57,181,74} 
\definecolor{violet}{RGB}{147,39,143} 
\newcommand{\darkblue}{\color{darkblue}} 
\newcommand{\defn}[1]{\textsl{\darkblue #1}} 
\newcommand{\para}[1]{\smallskip\noindent\uline{#1.}} 

\usepackage{todonotes}

\newcommand{\meet}{\wedge} 
\newcommand{\join}{\vee} 
\newcommandx{\projDown}[1][1={}]{\smash{\pi_\downarrow^{#1}}} 
\newcommandx{\projUp}[1][1={}]{\smash{\pi^\uparrow_{#1}}} 

\newcommandx{\Fan}[1][1=D]{\mathcal{F}_{#1}} 
\newcommand{\polytope}[1]{\mathds{#1}} 

\newcommand{\wigglyComplex}{\mathrm{WC}} 
\newcommand{\wigglyFlipGraph}{\mathrm{WFG}} 
\newcommand{\wigglyIncreasingFlipGraph}{\mathrm{WIFG}} 
\newcommand{\wigglyLattice}{\mathrm{WL}} 
\newcommand{\wigglyFan}{\mathrm{WF}} 
\newcommand{\wigglyhedron}{\polytope{W}} 
\newcommand{\Asso}{\polytope{A}\mathsf{sso}} 

\makeatletter
\def\l@part{\@tocline{1}{8pt}{0pc}{}{}}
\def\l@section{\@tocline{1}{4pt}{0pc}{}{}}
\makeatother
\let\oldtocpart=\tocpart
\renewcommand{\tocpart}[2]{\sc\large\oldtocpart{#1}{#2}}
\let\oldtocsection=\tocsection
\renewcommand{\tocsection}[2]{\bf\oldtocsection{#1}{#2}}
\let\oldtocsubsubsection=\tocsubsubsection
\renewcommand{\tocsubsubsection}[2]{\quad\oldtocsubsubsection{#1}{#2}}

\title{Wigglyhedra}

\thanks{AB was partially supported by the Australian Research Council DECRA grant DE240100447. 
VP was partially supported by the Spanish project PID2022-137283NB-C21 of MCIN/AEI/10.13039/501100011033 / FEDER, UE, by the Spanish--German project COMPOTE (AEI PCI2024-155081-2 \& DFG 541393733), by the Severo Ochoa and María de Maeztu Program for Centers and Units of Excellence in R\&D (CEX2020-001084-M), by the Departament de Recerca i Universitats de la Generalitat de Catalunya (2021 SGR 00697), and by the French--Austrian project PAGCAP (ANR-21-CE48-0020 \& FWF I 5788).}

\author{Asilata Bapat}
\address{The Australian National University}
\email{asilata.bapat@anu.edu.au}
\urladdr{\url{https://asilata.github.io}}

\author{Vincent Pilaud}
\address[Vincent Pilaud]{Universitat de Barcelona \& Centre de Recerca Matemàtica, Barcelona, Spain}
\email{vincent.pilaud@ub.edu}
\urladdr{\url{https://www.ub.edu/comb/vincentpilaud/}}

\begin{document}

\begin{abstract}
Motivated by categorical representation theory, we define the wiggly complex, whose vertices are arcs wiggling around $n+2$ points on a line, and whose faces are sets of wiggly arcs which are pairwise pointed and non-crossing. The wiggly complex is a $(2n-1)$-dimensional pseudomanifold, whose facets are wiggly pseudotriangulations. We show that wiggly pseudotriangulations are in bijection with wiggly permutations, which are permutations of~$[2n]$ avoiding the patterns~$(2j-1) \cdots i \cdots (2j)$ for~$i < 2j-1$ and~$(2j) \cdots k \cdots (2j-1)$ for~$k > 2j$. These permutations define the wiggly lattice, an induced sublattice of the weak order. We then prove that the wiggly complex is isomorphic to the boundary complex of the polar of the wigglyhedron, for which we give explicit and simple vertex and facet descriptions. Interestingly, we observe that any Cambrian associahedron is normally equivalent to a well-chosen face of the wigglyhedron. Finally, we recall the correspondence of wiggly arcs with objects in a category, and we develop categorical criteria for a subset of wiggly arcs to form a face of the wiggly complex.
\end{abstract}

\maketitle

\centerline{\includegraphics[scale=1.7]{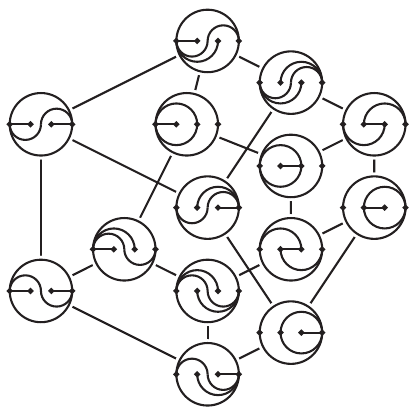}}

\tableofcontents


\section{Introduction}

In this paper we consider wiggly pseudotriangulations, which are built from arcs wiggling around points on a line.
We study their combinatorics in detail, construct polyhedral realizations of their flip graph, and interpret them in terms of categorical representation theory.

\bigskip
\para{Combinatorics}
Fix $n \ge 1$ and consider the~$n+2$ points on the horizontal axis labelled~${0, \dots, n+1}$.
A \defn{wiggly arc} is an abscissa monotone curve starting at some point~$i$, ending at some point~$j$ (with~${0 \le i < j \le n+1}$), and passing above or below each point~$i+1, \dots, j-1$.
Two wiggly arcs are \defn{compatible} if they are pointed (the left endpoint of one is not the right endpoint of the other) and non-crossing (the interiors of the curves are disjoint).
A \defn{wiggly dissection} is a collection of pairwise compatible wiggly arcs, and a \defn{wiggly pseudotriangulation} is an inclusion maximal wiggly dissection.
The \defn{wiggly complex}~$\wigglyComplex_n$ is the simplicial complex of wiggly dissections (minus their external wiggly arcs).
We first establish the following result.

\begin{proposition*}
The wiggly complex~$\wigglyComplex_n$ is a $(2n-1)$-dimensional pseudomanifold without boundary.
\end{proposition*}

In particular, all wiggly pseudotriangulations have $2n-1$ internal wiggly arcs, each of which can be \defn{flipped}.
We actually describe the flip in terms of wiggly diagonals in wiggly pseudoquadrangles.
We also orient this flip graph in a canonical way, which we call the \defn{wiggly increasing flip graph}~$\wigglyIncreasingFlipGraph_n$.
The wiggly increasing flip graph on wiggly pseudotriangulations of $4$ points is the front page illustration of this paper (throughout the paper, we implicitly orient the graphs and posets from bottom to top).

We then make a small detour to permutations.
A \defn{wiggly permutation} is a permutation of~$[2n]$ avoiding the patterns~$(2j-1) \cdots i \cdots (2j)$ for~$i < 2j-1$ and~$(2j) \cdots k \cdots (2j-1)$ for~$k > 2j$.
Our main result on wiggly permutations is the following.

\begin{proposition*}
The wiggly permutations induce a sublattice of the weak order on permutations of~$[2n]$, which we call the \defn{wiggly lattice}~$\wigglyLattice_n$. Moreover, its Hasse diagram is regular of degree~$2n-1$.
\end{proposition*}

We then connect wiggly pseudotriangulations to wiggly permutations.

\begin{proposition*}
The increasing flip graph on wiggly pseudotriangulations~$\wigglyIncreasingFlipGraph_n$ is isomorphic to the Hasse diagram of the wiggly lattice~$\wigglyLattice_n$. In particular, the wiggly pseudotriangulations admit a lattice structure.
\end{proposition*}

\para{Polyhedral geometry}
We then switch to geometric realizations of the wiggly complex~$\wigglyComplex_n$.
Our constructions are inspired by similar constructions for generalized associahedra~\cite{HohlwegLangeThomas,HohlwegPilaudStella} of finite type cluster algebras~\cite{FominZelevinsky-ClusterAlgebrasI,FominZelevinsky-ClusterAlgebrasII}, and for gentle associahedra~\cite{PaluPilaudPlamondon-nonkissing} of support \mbox{$\tau$-tilting} complexes~\cite{AdachiIyamaReiten} on gentle algebras~\cite{ButlerRingel}.
Namely, we first construct a fan supported by certain $\b{g}$-vectors associated to the wiggly arcs.

\begin{theorem*}
The cones generated by the $\b{g}$-vectors of wiggly pseudodissections form a complete simplicial fan, called \defn{wiggly fan}~$\wigglyFan_n$, realizing the wiggly complex~$\wigglyComplex_n$.
\end{theorem*}

The proof of this result is based on a fine understanding of the linear dependences between \mbox{$\b{g}$-vectors}.
We further use this understanding to find a height function~$\kappa$ on wiggly arcs which satisfies all wall-crossing inequalities of the wiggly fan~$\wigglyFan_n$.
Without entering the details of the definition of~$\kappa$ at the moment, let us just mention that it is obtained by summing some compatibility degrees of wiggly arcs, analogously to the constructions for cluster algebras~\cite{HohlwegPilaudStella} and gentle algebras~\cite{PaluPilaudPlamondon-nonkissing}.
Using the $\b{c}$-vectors of each pseudotriangulation, defined to form the dual basis of the $\b{g}$-vectors, we obtain our main result.

\begin{theorem*}
The wiggly fan~$\wigglyFan_n$ is the normal fan of a simplicial $(2n-1)$-dimensional polytope, called the \defn{wigglyhedron}~$\wigglyhedron_n$, and defined equivalently~as
\begin{itemize}
\item the intersection of the halfspaces~$\set{\b{x} \!\in\! \HH_{2n}\!}{\!\dotprod{\b{g}(\alpha)}{\b{x}} \!\le\! \kappa(\alpha)}$ for all internal wiggly arcs~$\alpha$,
\item the convex hull of the points~$\b{p}(T) \eqdef \sum\limits_{\alpha \in T} \kappa(\alpha) \, \b{c}(\alpha, T)$ for all wiggly pseudotriangulations~$T$.
\end{itemize}
\end{theorem*}

Moreover, as for the classical permutahedron and associahedron, we recover our lattice structure by a suitable linear orientation of the graph of the wigglyhedron.

\begin{proposition*}
The Hasse diagram of the wiggly lattice~$\wigglyLattice_n$ (or equivalently, the wiggly increasing flip graph~$\wigglyIncreasingFlipGraph_n$) is isomorphic to the graph of the wigglyhedron~$\wigglyhedron_n$ oriented in a suitable direction.
\end{proposition*}

\para{Further topics}
The wigglyhedron has interesting connections in lattice theory and discrete geometry.
We first connect the wigglyhedron to the (type~$A$) Cambrian world.

\begin{theorem*}
For any~$\delta \in \{-,+\}^n$,
\begin{itemize}
\item the $\delta$-Cambrian lattice of~\cite{Reading-CambrianLattices} is an interval in the wiggly lattice~$\wigglyLattice_n$,
\item the $\delta$-Cambrian fan of~\cite{ReadingSpeyer} is a link in the wiggly fan~$\wigglyFan_n$,
\item the $\delta$-associahedron of~\cite{HohlwegLange} is normally equivalent to a face of the wigglyhedron~$\wigglyhedron_n$.
\end{itemize}
\end{theorem*}

We then define the wiggly complex~$\wigglyComplex_P$ of an arbitrary planar point set~$P$ (not necessarily aligned along the horizontal axis, nor necessarily in general position).
Our definition recovers:
\begin{itemize}
\item when $P$ consists of $n+2$ collinear points, the wiggly complex~$\wigglyComplex_n$ considered in this~paper,
\item when~$P$ is in general position, the pseudodissection complex of~$P$ considered in discrete and computational geometry~\cite{PocchiolaVegter,RoteSantosStreinu-pseudotriangulations}, and realized by the pseudotriangulation polytope in~\cite{RoteSantosStreinu-polytope} using rigidity and expansive motions.
\end{itemize}
These two families of examples motivate the following conjecture.

\begin{conjecture*}
For any point set~$P$ in the plane, the wiggly complex~$\wigglyComplex_P$ is the boundary complex of a simplicial polytope.
\end{conjecture*}

We also present a few open questions related to this conjecture, concerning in particular the possible interpretation of the wiggly complex~$\wigglyComplex_P$ in terms of dual pseudoline arrangements, and the extension to multi wiggly complexes and arbitrary finite Coxeter groups.

\pagebreak
\para{Categorical interpretation}
Each wiggly arc on \(n+2\) points corresponds to an object of an abelian category, such that morphisms and extensions between objects in the category are captured by minimal intersection numbers between the wiggly arcs.
In fact, this correspondence holds for more general curves, which then correspond to objects of a larger triangulated category~\cite{kho.sei:02}.

A general curve can be decomposed into wiggly arcs simply by cutting along its U-turns.
We interpret this operation categorically as the decomposition of an object of the triangulated category into its cohomology pieces, from which we deduce the following.

\begin{proposition*}
The set of wiggly arcs corresponding to the cohomology pieces of an object represented by a general arc is a face of the wiggly complex.
\end{proposition*}

Furthermore, we characterize categorically when two objects in the abelian category are compatible, that is, their corresponding wiggly arcs are compatible.

\begin{proposition*}
Two wiggly arcs are compatible if and only if there are no extensions between the corresponding objects in the abelian category.
\end{proposition*}


\section{Wiggly pseudotriangulations and wiggly permutations}
\label{sec:combinatorics}

Fix an integer~$n \ge 1$.
In this section, we define two combinatorial families, the wiggly pseudotriangulations of~$n+2$ points (\cref{subsec:wigglyPseudotriangulations}) and the wiggly permutations of~$[2n]$ (\cref{subsec:wigglyPermutations}), and we prove that they are in bijection (\cref{subsec:bijection}).
These two perspectives enable us to endow these combinatorial families with additional structures: the wiggly pseudotriangulations are actually the facets of the wiggly complex, while the wiggly permutations are actually the elements of the wiggly lattice.
We conclude with a few questions and conjectures on graph properties of the wiggly flip graph (\cref{subsec:graphProperties}).


\subsection{Wiggly pseudotriangulations and the wiggly complex}
\label{subsec:wigglyPseudotriangulations}

We start with wiggly pseudodissections of $n+2$ points on a line.
The following two definitions are illustrated in \cref{fig:incompatible}.

\begin{definition}
A \defn{wiggly arc} is a quadruple $\alpha \eqdef (i,j,A,B)$ where $0 \le i < j \le n+1$ and the sets~$A$ and~$B$ form a partition of~${]i,j[} \eqdef \{i+1, \dots, j-1\}$.
We represent $\alpha$ by an abscissa monotone curve wiggling around the horizontal axis, starting at point~$i$, ending at point~$j$, and passing above the points of~$A$ and below the points of~$B$.
The wiggly arcs~$\alpha_\mathrm{top} \eqdef (0, n+1, [n], \varnothing)$ and~${\alpha_\mathrm{bot} \eqdef (0, n+1, \varnothing, [n])}$ are called \defn{external}, all other wiggly arcs are called \defn{internal}.
\end{definition}

\begin{remark}
\label{rem:numberWigglyArcs}
The number of wiggly arcs is~$\sum_{\ell = 0}^n 2^\ell (n+1-\ell) = 2^{n+2}-n-3$.
\end{remark}

\begin{definition}
\label{def:compatible}
Two wiggly arcs~$(i,j,A,B)$ and~$(i',j',A',B')$ are 
\begin{itemize}
\item \defn{non pointed} if~$i = j'$ or~$i' = j$,
\item \defn{crossing} if~$(A \cap B') \cup (\{i,j\} \cap B') \cup (A \cap \{i',j'\}) \ne \varnothing \ne (A' \cap B) \cup (\{i',j'\} \cap B) \cup (A' \cap \{i,j\})$, that is, if the corresponding curves cross (\ie any isotopy representatives must cross),
\item \defn{incompatible} if they are non pointed or crossing, and \defn{compatible} otherwise.
\end{itemize}
\begin{figure}[b]
\centerline{\includegraphics[scale=1.3]{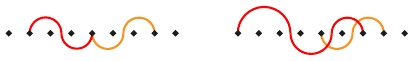}}
\caption{Some incompatible wiggly arcs: non pointed (left) and crossing (right).}
\label{fig:incompatible}
\end{figure}
\end{definition}

\begin{remark}
We want to stress out that our notion of compatibility on wiggly arcs of \cref{def:compatible} might seem similar to, but fundamentally differs from the notion of compatibility used by N.~Reading to define noncrossing arc diagrams~\cite{Reading-arcDiagrams}.
Namely, our wiggly arcs and our crossing condition coincide with that of~\cite{Reading-arcDiagrams}.
However, we forbid an arc to start where another arc ends but allow two arcs to start or end at the same point,
while N.~Reading does precisely the opposite in~\cite{Reading-arcDiagrams}.
\end{remark}

We then consider pairwise compatible sets of wiggly arcs.
The following two definitions are illustrated~in~\cref{fig:pseudodissections}.

\begin{definition}
A \defn{wiggly pseudodissection}~$D$ is a set of pairwise compatible wiggly arcs which contains the exterior wiggly arcs~$\alpha_\mathrm{top}$ and~$\alpha_\mathrm{bot}$. We denote by~$D^\circ \eqdef D \ssm \{\alpha_\mathrm{top}, \alpha_\mathrm{bot}\}$.
\begin{figure}
\centerline{\includegraphics[scale=1.5]{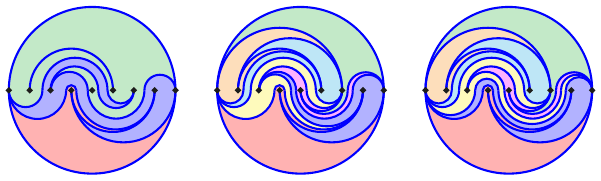}}
\caption{Three wiggly pseudodissections for~$n = 7$. The left one has three wiggly cells, of respective degrees $3$ (red), $8$ (green), and $4$ (blue). The other two have $7$ wiggly pseudotriangles, and are obtained from each other by flipping the wiggly arc separating their yellow and dark blue wiggly pseudotriangles.}
\label{fig:pseudodissections}
\end{figure}
\end{definition}


\begin{definition}
A wiggly pseudodissection~$D$ decomposes the digon bounded by~$\alpha_\mathrm{top}$ and~$\alpha_\mathrm{bot}$ into cells, that we call \defn{wiggly cells}.
The \defn{degree}~$\delta_c$ of a wiggly cell~$c$ is~$\delta_c \eqdef \beta_c/2+2\kappa_c-1$ where~$\beta_c$ is the number of wiggly arcs on the boundary of~$c$ (counted twice if they appear twice along the boundary of~$c$), and~$\kappa_c$ is the number of connected components of the boundary of~$c$.
Two consecutive wiggly arcs along the boundary of~$c$ define a \defn{corner} (resp.~a \defn{hinge}) of~$c$ if they bound a convex (resp.~concave) angle of~$c$.
Note that two corners of~$c$ can coincide at the same point~$i \in [n]$.
\end{definition}

\begin{remark}
\label{rem:degree}
Consider a wiggly cell~$c$ in a wiggly pseudodissection~$D$.
Note that~$\beta_c$ is always even as the arcs of~$D$ are pairwise pointed, so that~$\beta_c/2$ is an integer.
Moreover, we have~$\beta_c \ge 2$ and~$\kappa_c \ge 1$, where at least one of the two inequalities is strict (as otherwise, the two wiggly arcs bounding~$c$ would coincide).
We thus obtain on the one hand that~$\delta_c = \beta_c/2+2\kappa_c-1 \ge 3$ with equality if and only if~$\beta_c = 4$ and~$\kappa_c = 1$, and on the other hand that~$\beta_c/2+\kappa_c-1 \ge 2$.
\end{remark}

\begin{definition}
A \defn{wiggly pseudotriangle} (resp.~\defn{pseudoquadrangle}) is a wiggly cell of degree~$3$ (resp.~$4$).
\end{definition}

\begin{remark}
\label{rem:descriptionPseudotrianglesPseudoquadrangles}
As illustrated in \cref{fig:pseudotriangles}, a wiggly pseudotriangle~$t$ has no point in its interior and $4$ points on its boundary, $3$ corners and $1$ hinge (two corners can coincide at the same point~$i \in [n]$).
As illustrated in \cref{fig:pseudoquadrangles}, a wiggly pseudoquadrangle~$q$ can have
\begin{enumerate}
\item either one point in its interior and~$2$ points on its boundary, both corners,
\item or no point in its interior and $6$ points on its boundary, $4$ corners and $2$ hinges (two corners can coincide at the same point~$i \in [n]$). We then distinguish three situations, depending on the position of the two pairs~$p, p'$ of wiggly arcs incident to the two hinges of~$c$, as illustrated in \cref{fig:pseudoquadrangles} (where~$p$ and~$p'$ are colored in red and orange):
\begin{enumerate}[(a)]
\item those where $p$ and $p'$ are not consecutive along~$c$, and do not overlap vertically,
\item those where $p$ and $p'$ are not consecutive along~$c$, but overlap vertically,
\item those where~$p$ and~$p'$ are consecutive along~$c$.
\end{enumerate}
\end{enumerate}
\begin{figure}
\centerline{\includegraphics[scale=1.3]{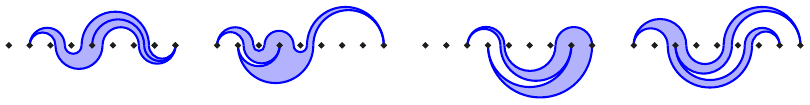}}
\caption{Four wiggly pseudotriangles.}
\label{fig:pseudotriangles}
\end{figure}
\begin{figure}
\centerline{\includegraphics[scale=1.3]{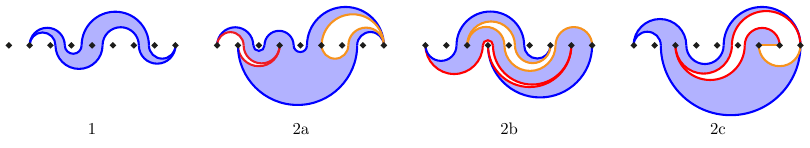}}
\vspace{-.3cm}
\caption{Four types of wiggly pseudoquadrangles described in \cref{rem:descriptionPseudotrianglesPseudoquadrangles}.}
\label{fig:pseudoquadrangles}
\end{figure}
\end{remark}

\begin{proposition}
\label{prop:wigglyPseudotriangulations}
Any wiggly pseudodissection~$D$ contains at most~$2n-1$ internal wiggly arcs and at most~$n$ wiggly cells.
Moreover, the following assertions are equivalent:
\begin{enumerate}[(i)]
\item all wiggly cells of~$D$ are wiggly pseudotriangles,
\item $D$ contains~$n$ wiggly cells and is connected,
\item $D$ contains $2n-1$ internal wiggly arcs,
\item $D$ is an inclusion maximal wiggly pseudodissection.
\end{enumerate}
We then say that~$D$ is a \defn{wiggly pseudotriangulation}.
\end{proposition}

\begin{proof}
Consider a wiggly pseudodissection~$D$ with~$\iota$ internal wiggly arcs, $\gamma$ wiggly cells, and~$\kappa$~connected components (where $D$ is viewed as the edge set of the graph with vertex set given by the marked points $0$ to $n+1$).
Euler's formula gives~$(n+2)-(\iota+2)+(\gamma+1)=(\kappa+1)$, hence~$n - \iota + \gamma = \kappa$.
Moreover, denoting by~$C$ the set of wiggly cells of~$D$, we have~$\sum_{c \in C} \beta_c/2 = 1+\iota$ since each external (resp.~internal) wiggly arc bounds precisely one (resp.~two) wiggly cell of~$D$ (we count a wiggly arc~$\alpha$ with multiplicity~$2$ if it appears twice along the boundary of~$c$), and~${\sum_{c \in C} \big( \kappa_c-1 \big) = \kappa-1}$ since the boundary of each internal connected component of~$D$ is one of the internal connected components of the boundary of precisely one wiggly cell of~$D$.
As~$2 \le \beta_c/2+\kappa_c-1$ for any wiggly cell~$c$ by \cref{rem:degree}, we thus obtain that~$2\gamma \le \sum_{c \in C} (\beta_c/2+\kappa_c-1) = \iota+\kappa = n+\gamma$.
We conclude that~$\gamma \le n$ and~$\iota = n+\gamma-\kappa \le 2n-1$.


To prove the second part of the statement, we show that~(i) $\Rightarrow$ (ii) $\Rightarrow$ (iii) $\Rightarrow$ (iv) $\Rightarrow$ (i).

\para{(i) $\Rightarrow$ (ii)}
For any wiggly cell~$c$, we have~$\delta_c = 3$, so that~$\beta_c = 4$ and~$\kappa_c = 1$ by \cref{rem:degree}, hence~$2 = \beta_c/2+\kappa_c-1$.
Therefore~$\kappa = 1$ ($D$ is connected) and~$2\gamma = \iota+1 = n+\gamma$, so that~$\gamma=n$.

\para{(ii) $\Rightarrow$ (iii)}
As $\gamma=n$ and~$\kappa=1$, we have~$\iota = n+\gamma-\kappa = 2n-1$.

%

\para{(iii) $\Rightarrow$ (iv)}
If~$D$ already contains~$2n-1$ internal wiggly arcs, it has maximal cardinality among all wiggly pseudodissections, hence it is certainly inclusion maximal.

\para{(iv) $\Rightarrow$ (i)}
We now consider a wiggly pseudodissection~$D$ containing a wiggly cell~$c$ which is not a wiggly pseudotriangle, and we prove that~$D$ is not inclusion maximal.
Assume first that the external boundary of~$c$ is a digon bounded by the wiggly arcs~$\alpha \eqdef (i, j, A, B)$ and~$\alpha' \eqdef (i, j, A', B')$.
Denote by~$m$ the maximum of the points~$(A \symdif A') \cup (B \symdif B')$ located inside this digon (where~$\symdif$ denotes the symmetric difference).
Then the wiggly arc~$\beta \eqdef \big( i, m, {]i,m[} \cap (A \cup A'), {]i,m[} \cap (B \cap B') \big)$ is compatible with~$D$ (any wiggly arc incompatible with~$\beta$ crosses~$\alpha$ or~$\alpha'$) and is not in~$D$ (it connects the boundary of~$c$ with one of its internal connected components).

Assume now that the external boundary of~$c$ is a not digon.
Consider a hinge~$j$ of~$c$.
Assume by symmetry that the boundary of~$c$ contains two wiggly arcs~$\alpha \eqdef (i, j, A, B)$ and~$\alpha' \eqdef (i', j, A', B')$, such that~$\alpha$ is below~$\alpha'$. 
Let~$\gamma \eqdef (k, \ell, X, Y)$ denote the uppermost wiggly arc of the boundary of~$c$ such that~$j \in Y$ (it exists since~$j$ is a  hinge of~$c$), and let~$X' \eqdef {]j,\ell[} \cap X$ and~$Y' \eqdef {]j,\ell[} \cap Y$.
Then the two wiggly arcs~$\beta \eqdef \big( i, \ell, A \cup X', B \cup \{j\} \cup Y')$ and~$\beta' \eqdef \big( i', \ell, A'  \cup \{j\} \cup X', B \cup Y')$ are both compatible with~$D$ (any wiggly arc incompatible with~$\beta$ is incompatible with~$\alpha$ or~$\gamma$, and similarly for~$\beta'$) and cannot both belong to~$D$ (otherwise, $c$ would be a wiggly pseudotriangle bounded by~$\alpha, \alpha', \gamma', \gamma$).
\end{proof}

\pagebreak
\begin{remark}
Note that in the proof of \cref{prop:wigglyPseudotriangulations}, we actually obtained that
\[
\sum_{c \in C} (\delta_c-3) = \sum_{c \in C} \beta_c/2 + 2 \sum_{c \in C} (\kappa_c-1) - 2\gamma = (1+\iota) + 2(\kappa-1) - 2(\iota+\kappa-n) = 2n-1-\iota.
\]
\end{remark}

\begin{example}
\label{exm:allSmallWigglyPseudotriangulations}
For~$n = 1$, the $2$ wiggly pseudotriangulations of $3$ points are \smash{\raisebox{-.3cm}{\includegraphics[scale=.8]{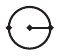}}} and \smash{\raisebox{-.3cm}{\includegraphics[scale=.8]{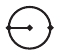}}}.
For~${n = 2}$, the $14$ wiggly pseudotriangulations of $4$ points are \\[.2cm]
\centerline{\includegraphics[scale=.8]{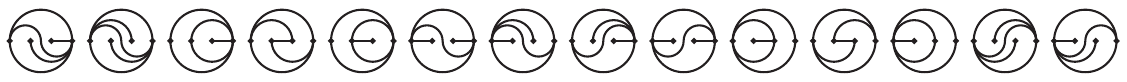}}
\end{example}

\begin{example}
\label{exm:specialWigglyPseudotriangulations}
The following wiggly pseudotriangulations are illustrated in \cref{fig:specialWigglyPseudotriangulations}:
\begin{align*}
T_\downarrow & \eqdef \set{(i-1, n+1, \varnothing, {[i,n]})}{i \in [n+1]} \cup \set{(0, n+1, [i], {]i,n]})}{i \in [n]}, \\
T_\uparrow & \eqdef \set{(0, i, \varnothing, {[i-1]})}{i \in [n+1]} \cup \set{(0, n+1, {]i,n]}, [i])}{i \in [n]}, \\
T_\leftarrow & \eqdef \{(0, 1, \varnothing, \varnothing)\} \cup \set{(0, i+1, \varnothing, [i])}{i \in [n]} \cup \set{(0, i+1, [i], \varnothing)}{i \in [n]}, \\
T_\rightarrow & \eqdef \{(n, n+1, \varnothing, \varnothing)\} \cup \set{(i-1, n+1, \varnothing, [i,n])}{i \in [n]} \cup \set{(i-1, n+1, [i,n], \varnothing)}{i \in [n]}.
\end{align*}
\begin{figure}[h]
\centerline{\includegraphics[scale=1.3]{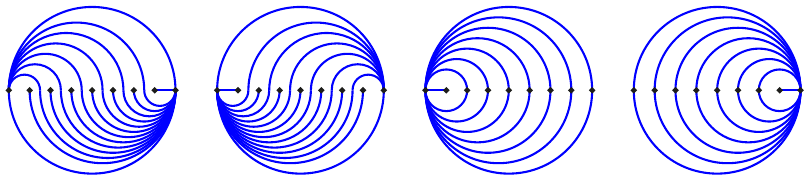}}
\caption{The wiggly pseudotriangulations~$T_\downarrow$ (left), $T_\uparrow$ (middle left), $T_\leftarrow$ (middle right), and $T_\rightarrow$ (right) of \cref{exm:specialWigglyPseudotriangulations} for~$n = 7$.}
\label{fig:specialWigglyPseudotriangulations}
\end{figure}
\end{example}

\begin{remark}
\label{rem:numberWigglyPseudotriangulations}
The number~$wp_n$ of wiggly pseudotriangulations of~$n+2$ points is given by
\[
\begin{array}{c|ccccccccc}
n & 1 & 2 & 3 & 4 & 5 & 6 & 7 & 8 & \dots \\
\hline
wp_n & 2 & 14 & 176 & 3232 & 78384 & 2366248 & 85534176 & 3602770400 & \dots
\end{array}
\]
To compute~$wp_n$, denote by $wp_n(x)$ the polynomial where the coefficient of~$x^i$ is the number of wiggly pseudotriangulations of~$n+2$ points with~$i$ \defn{final} internal wiggly arcs (\ie ending at the last point~$n+1$).
For instance,
\begin{align*}
	wp_1(x) & = x + 1, \\
	wp_2(x) & = 3 x^3 + 5 x^2 + 4 x + 2, \\
	wp_3(x) & = 15 x^5 + 35 x^4 + 44 x^3 + 40 x^2 + 28 x + 14, \\
	wp_4(x) & = 105 x^7 + 315 x^6 + 520 x^5 + 630 x^4 + 620 x^3 + 514 x^2 + 352 x + 176.
\end{align*}
We invite the reader to check the expressions of~$wp_1(x)$ and~$wp_2(x)$ with \cref{exm:allSmallWigglyPseudotriangulations}.
We use this additional variable~$x$ to obtain a recursive formula for~$wp_n(x)$.
Indeed, any wiggly pseudotriangulation of~$n+2$ points can be obtained from a wiggly pseudotriangulation of~$n+1$ points by inserting a single wiggly pseudotriangle with a hinge at~$n$.
Each wiggly pseudotriangulation of~$n+1$ points with $i$ final internal wiggly arcs contributes to:
\begin{itemize}
\item $i+2$ wiggly pseudotriangulations of~$n+2$ points with $i+2$ final internal wiggly arcs, containing the wiggly arc~$(n, n+1, \varnothing, \varnothing)$,
\item $j+1$ wiggly pseudotriangulations of~$n+2$ points with $j$ final internal wiggly arcs, not containing the wiggly arc~$(n, n+1, \varnothing, \varnothing)$,  for all~$0 \le j \le i+1$.
\end{itemize}
\pagebreak
\cref{fig:pseudotriangulationsCounting} illustrates this decomposition.
\begin{figure}
\centerline{\includegraphics[scale=1.3]{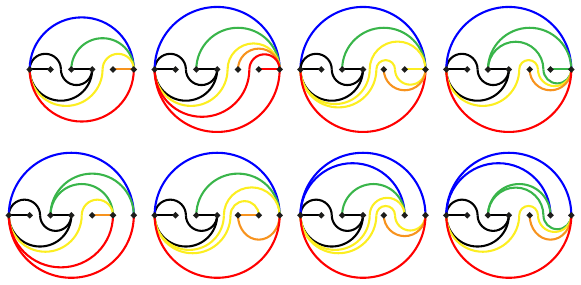}}
\caption{Some wiggly pseudotriangulations of~$7$ points obtained from the top left wiggly pseudotriangulation of~$6$ points by inserting a wiggly pseudotriangle with hinge at~$5$ as described in \cref{rem:numberWigglyPseudotriangulations}. Those on the first line contain the wiggly arc~$(5, 6, \varnothing, \varnothing)$ and are thus obtained by inserting the head of the wiggly pseudotriangle with hinge at~$5$ in one of the final wiggly arcs. Those on the second line do not contain the wiggly arc~$(5, 6, \varnothing, \varnothing)$ and are thus obtained by inserting the two legs of the wiggly pseudotriangle with hinge at~$5$ in two of the final wiggly arcs. The coloring of the wiggly arcs is supposed to help visualizing in which wiggly arcs the head or legs of the new wiggly pseudotriangle were inserted.}
\label{fig:pseudotriangulationsCounting}
\end{figure}
This translates to the recursive formula 
\begin{align*}
wp_{n}(x) & = \overbracket[.5pt]{x \cdot \frac{\partial}{\partial x} \big( x^2 wp_{n-1}(x) \big)}^{\substack{\text{wiggly pseudotriangulations} \\ \text{containing } (n, n+1, \varnothing, \varnothing)}} + \overbracket[.5pt]{\frac{\partial}{\partial x}\Big( \smash{\frac{wp_{n-1}(1)-x^3wp_{n-1}(x)}{1-x}} \Big)}^{\substack{\qquad\; \text{wiggly pseudotriangulations} \qquad\;  \\ \text{not containing } (n, n+1, \varnothing, \varnothing)}} \\
& = \frac{1}{(1 - x)^2} \big( wp_{n-1}(1) + x^2 (2 x^2 - 2 x - 1) wp_{n-1}(x) + x^4 (x - 1) wp_{n-1}’(x) \big).
\end{align*}
This enables us to quickly compute~$wp_n(x)$ and we obtain the numbers of wiggly pseudotriangulations~$wp_n$ by evaluating~$wp_n(x)$ at~$x = 1$.
\end{remark}

Our next statement prepares for the definition of flips in wiggly pseudotriangulations, and is illustrated in \cref{fig:diagonalsPseudoquadrangles}.

\begin{definition}
A \defn{wiggly diagonal} of a wiggly cell~$c$ is a wiggly arc inside~$c$ and compatible with all wiggly arcs of~$c$.
\end{definition}

\begin{proposition}
\label{prop:diagonalsPseudoquadrangle}
Any wiggly pseudoquadrangle has exactly two wiggly diagonals.
Moreover, these two wiggly diagonals either cross precisely once, or are non pointed.
\end{proposition}

\begin{proof}
Depending on the four types of pseudoquadrangles described in \cref{rem:descriptionPseudotrianglesPseudoquadrangles}:
\begin{enumerate}
\item the two diagonals are the two wiggly arcs inside~$q$, incident to the point in the interior~of~$q$,
\item[(2a)] the two diagonals connect the two pairs of opposite corners of~$q$,
\item[(2b)] one of two diagonals connects two opposite corners of~$q$ while the other connects the two hinges~of~$q$,
\item[(2c)] one of the diagonals connects two opposite corners of~$q$ while the other connects an hinge of~$q$ to a corner of~$q$.
\end{enumerate}
See \cref{fig:diagonalsPseudoquadrangles} where the two diagonals appear in red and orange.
\begin{figure}
\centerline{\includegraphics[scale=1.3]{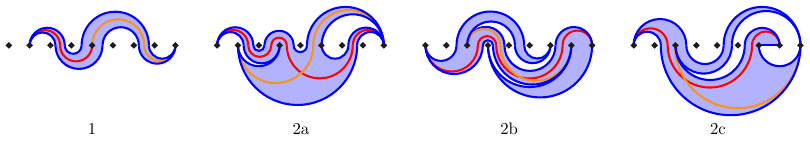}}
\caption{The two diagonals in each of the four types of wiggly pseudoquadrangles.} 
\label{fig:diagonalsPseudoquadrangles}
\end{figure}
\end{proof}


We are now ready to consider the wiggly flip graph and the wiggly complex.
The following definition is illustrated in \cref{fig:wigglyComplex}\,(right) when~$n = 2$.
\begin{figure}
\centerline{\includegraphics[scale=1.1]{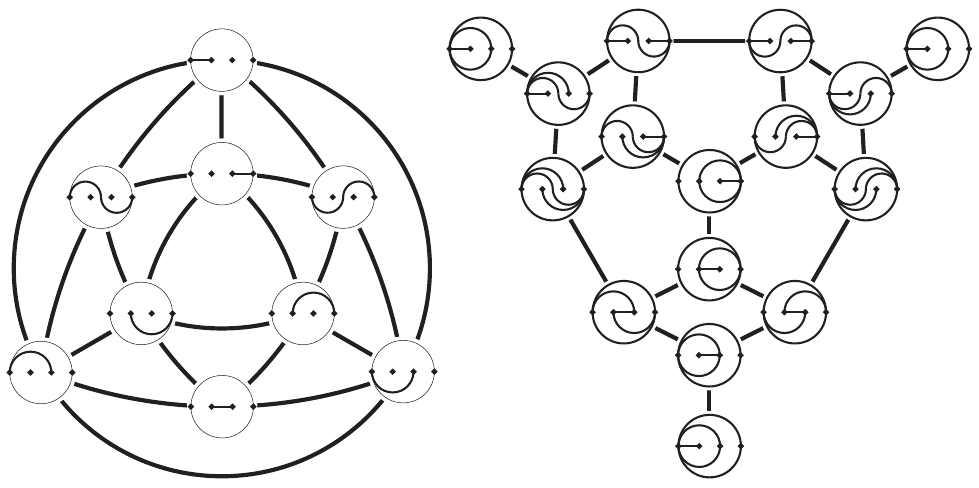}}
\caption{The wiggly complex~$\wigglyComplex_2$ (left) and the wiggly flip graph~$\wigglyFlipGraph_2$ (right). We have embedded the wiggly complex~$\wigglyComplex_2$ in the plane, so that all triangles (including the external one) form a face of~$\wigglyComplex_2$.}
\label{fig:wigglyComplex}
\end{figure}

\begin{definition}
\label{def:wigglyFlipGraph}
The \defn{wiggly flip graph}~$\wigglyFlipGraph_n$ is the graph with a vertex for each wiggly pseudotriangulation, and with an edge between two wiggly pseudotriangulations~$T$ and~$T'$ if there are wiggly arcs~$\alpha \in T$ and~$\alpha' \in T'$ such that~$T \ssm \{\alpha\} = T' \ssm \{\alpha'\}$.
\end{definition}

\begin{proposition}
\label{prop:wigglyFlipGraph}
The wiggly flip graph~$\wigglyFlipGraph_n$ is regular of degree~$2n-1$ and connected.
\end{proposition}

\begin{proof}
Consider any internal wiggly arc~$\alpha$ of a wiggly pseudotriangulation~$T$.
Then~${Q \eqdef T \ssm \{\alpha\}}$ contains~$n-2$ wiggly pseudotriangles and $1$ pseudoquadrangle~$q$.
By \cref{prop:diagonalsPseudoquadrangle}, $q$ admits precisely two wiggly diagonals, namely $\alpha$ and another one~$\alpha'$.
As $\alpha'$ is inside~$q$ and compatible with~$c$, it is also compatible with~$Q$.
We obtain that~$T' \eqdef Q \cup \{\alpha'\}$ is the only wiggly pseudotriangulation other than~$T$ containing~$Q$.
This shows that any internal wiggly arc in any wiggly pseudotriangulation can be flipped, hence that the wiggly flip graph is regular of degree~$2n-1$ by \cref{prop:wigglyPseudotriangulations}\,(iii).

There are various possible proofs for the connectedness.
For instance, it will follow easily from \cref{prop:bijection}.
To avoid relying on this forward reference, we sketch an alternative direct proof by induction on~$n$.
We claim that if a wiggly pseudotriangulation has a non-zero number of final internal wiggly arcs (\ie ending at the last point~$n+1$), then at least one of them can be flipped to a non-final wiggly arc.
This implies that any wiggly pseudotriangulation can be transformed by flips to a wiggly pseudotriangulation with no final internal wiggly arc.
All wiggly pseudotriangulations with no final internal wiggly arc contain the arcs~$(0,n,\varnothing,[n-1])$ and~$(0,n,[n-1],\varnothing)$, since the latter are compatible with all internal wiggly arcs but the final ones.
Hence, they induce a subgraph of~$\wigglyFlipGraph_n$ isomorphic to~$\wigglyFlipGraph_{n-1}$, which is connected by induction.
We conclude that~$\wigglyFlipGraph_n$ is itself connected.

To see the claim, let~$i$ be maximal such that~$T$ contains two distinct wiggly arcs starting at~$i$ and ending at~$n+1$ (it is well-defined as we have the external wiggly arcs~$\alpha_\mathrm{top}$ and~$\alpha_\mathrm{bot}$ for~$i = 0$).
Consider a final internal wiggly arc~$\alpha$ of~$T$ starting weakly after~$i$, and let~$\alpha'$ be the wiggly arc obtained by flipping~$\alpha$ in~$T$.
If~$\alpha'$ is final, then by~\cref{prop:diagonalsPseudoquadrangle}, at least one pseudotriangle~$t$ of~$T$ incident to~$\alpha$ has two corners at~$n+1$.
Let~$h$ be the hinge of~$t$ and~$\beta, \beta'$ be the two wiggly arcs of~$t$ joining~$h$ to~$n+1$.
As~$i < h \le n$ the maximality of~$i$ imposes that~$\beta = \beta'$.
Flipping $\beta = \beta'$ in~$T$ thus yields a wiggly arc ending at~$h \le n$, hence a non-final arc.
\end{proof}

We will also consider an oriented version of the wiggly flip graph~$\wigglyFlipGraph_n$.
The following definition is illustrated in \cref{fig:incompatible2}.

\begin{definition}
\label{def:wigglyIncreasingFlipGraph}
Let~$\alpha \eqdef (i, j, A, B)$ and~$\alpha' \eqdef (i', j', A', B')$ be two wiggly arcs and~$T$ and~$T'$ be two wiggly pseudotriangulations, such that~$T \ssm \{\alpha\} = T' \ssm \{\alpha'\}$.
We say that the flip from~$T$ to~$T'$ is \defn{increasing} if 
\begin{itemize}
\item $\alpha$ starts where~$\alpha'$ ends (that is, $i = j'$), or 
\item $\alpha$ crosses~$\alpha'$ from north-west to south-east (that is, there exists~$1 \le k < \ell \le n$ such that $k \in (A \cap B') \cup (\{i\} \cap B') \cup (A \cap \{i'\})$ and~$\ell \in (A' \cap B) \cup (\{j'\} \cap B) \cup (A' \cap \{j\})$).
\end{itemize}
The \defn{wiggly increasing flip graph}~$\wigglyIncreasingFlipGraph_n$ is the directed graph with a vertex for each wiggly pseudotriangulation and an edge for each increasing~flip.
\begin{figure}
\centerline{\includegraphics[scale=1.3]{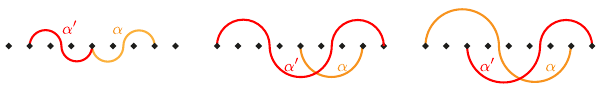}}
\caption{The wiggly arcs~$\alpha$ and~$\alpha'$ of \cref{def:wigglyIncreasingFlipGraph}.}
\label{fig:incompatible2}
\end{figure}
\end{definition}

We will need the following observation about exchangeable arcs, illustrated in \cref{fig:incompatible3}.

\begin{proposition}
\label{prop:uerp}
Let~$\alpha \eqdef (i, j, A, B)$ and~$\alpha' \eqdef (i', j', A', B')$ be two exchangeable wiggly arcs.
Then any wiggly arc compatible with both~$\alpha$ and~$\alpha'$ is also compatible with the wiggly~arcs:
\begin{itemize}
\item $\beta \eqdef (i, j', A \cup A', B \cup B' \cup \{j\})$ and~$\beta' \eqdef (i, j', A \cup A' \cup \{j\}, B \cup B')$ if~$\alpha$ and~$\alpha'$ are non pointed with~$j = i'$,
(In other words, ~$\beta$ and~$\beta'$ are the two wiggly arcs starting at~$\min(i,i')$ and ending at~$\max(j,j')$ which follow $\alpha$ on~$]i,j[$ and~$\alpha'$ on~$]i',j'[$.)
\item $\beta \eqdef \big( i', j, {]i,j'[} \ssm (B \cup B'), {]i,j'[} \cap (B \cup B') \big)$ and~$\beta' \eqdef \big( i, j', {]i',j[} \cap (A \cup A'), {]i',j[} \ssm (A \cup A') \big)$ if~$\alpha$ and~$\alpha'$ are crossing and $\alpha$ crosses~$\alpha'$ from north-west to south-east.
(In other words, $\beta$ is the wiggly arc starting at~$i'$ and ending at~$j$ which follows the lower hull of~$\alpha$ and~$\alpha'$, and~$\beta'$ is the wiggly arc starting at~$i$ and ending~at~$j'$ which follows the upper hull of~$\alpha$ and~$\alpha'$.)
\end{itemize}
(By exchanging~$\alpha$ and~$\alpha'$, similar statements hold if~$\alpha$ and~$\alpha'$ are non pointed with~$j = i'$, or if~$\alpha$ crosses~$\alpha'$ from south-east to north-west.)
\begin{figure}
\centerline{\includegraphics[scale=1.3]{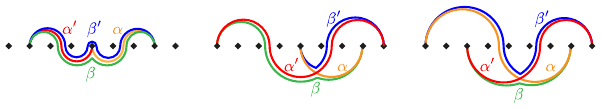}}
\caption{The wiggly arcs~$\beta$ and~$\beta'$ of \cref{prop:uerp}.}
\label{fig:incompatible3}
\end{figure}
\end{proposition}

\begin{proof}
Any wiggly arc non pointed with~$\beta$ would be non pointed with~$\alpha$ or~$\alpha'$ (as the endpoints of~$\beta$ are endpoints of~$\alpha$ or~$\alpha'$).
Any wiggly arc crossing~$\beta$ would be either non pointed or crossing~$\alpha$ or~$\alpha'$ (as~$\beta$ follows~$\alpha$ or~$\alpha'$, except at the point~$j = i'$ in the first situation).
\end{proof}

Finally, we define the wiggly complex~$\wigglyComplex_n$, illustrated in \cref{fig:wigglyComplex}\,(left) when~$n = 2$.

\begin{definition}
\label{def:wigglyComplex}
The \defn{wiggly complex}~$\wigglyComplex_n$ is the simplicial complex of pairwise pointed and non-crossing subsets of internal wiggly arcs.
In other words, it is the clique complex of the compatibility graph on internal wiggly arcs.
\end{definition}

\begin{proposition}
The wiggly complex~$\wigglyComplex_n$ is a flag $(2n-1)$-dimensional pseudomanifold without boundary.
\end{proposition}

\begin{proof}
It is a flag simplicial complex because it is a clique complex.
It is pure of dimension $2n-1$ by the implication (iv) $\Rightarrow$ (iii) of \cref{prop:wigglyPseudotriangulations}.
Finally, it is a pseudomanifold without boundary by \cref{prop:wigglyFlipGraph}.
\end{proof}

\begin{remark}
$\wigglyComplex_1$ contains two isolated points.
$\wigglyComplex_2$ is a simplicial $3$-dimensional associahedron (this coincidence fails for~$n > 2$).
Here are the first few $f$-vectors:
\begin{align*}
f(\wigglyComplex_1) & = (1, 2), \\
f(\wigglyComplex_2) & = (1, 9, 21, 14), \\
f(\wigglyComplex_3) & = (1, 24, 154, 396, 440, 176), \\
f(\wigglyComplex_4) & = (1, 55, 729, 4002, 10930, 15684, 11312, 3232), \\
f(\wigglyComplex_5) & = (1, 118, 2868, 28110, 140782, 400374, 673274, 662668, 352728, 78384),
\intertext{and the first few $h$-vectors:}
h(\wigglyComplex_1) & = (1, 1), \\
h(\wigglyComplex_2) & = (1, 6, 6, 1), \\
h(\wigglyComplex_3) & = (1, 19, 68, 68, 19, 1), \\
h(\wigglyComplex_4) & = (1, 48, 420, 1147, 1147, 420, 48, 1), \\
h(\wigglyComplex_5) & = (1, 109, 1960, 11254, 25868, 25868, 11254, 1960, 109, 1).
\end{align*}
Note that the number of vertices and facets of~$\wigglyComplex_n$ were already discussed in \cref{rem:numberWigglyArcs,rem:numberWigglyPseudotriangulations}.
\end{remark}

\begin{remark}
\label{rem:wigglyComplexAutomorphism}
The vertical and horizontal reflections provide two obvious automorphisms of the wiggly complex~$\wigglyComplex_n$ (and thus of the wiggly flip graph~$\wigglyFlipGraph_n$).
Moreover, the vertical reflection is an anti-automorphism of the wiggly increasing flip graph~$\wigglyIncreasingFlipGraph_n$ (while the horizontal reflection is not).
\end{remark}


\subsection{Wiggly permutations and the wiggly lattice}
\label{subsec:wigglyPermutations}

We now consider wiggly permutations of~$[2n]$, defined as follows.

\begin{definition}
\label{def:wigglyPermutation}
A \defn{wiggly permutation} is a permutation of~$[2n]$ which avoids the patterns
\begin{itemize}
\item $(2j-1) \cdots i \cdots (2j)$ for~$j \in [n]$ and~$i < 2j-1$,
\item $(2j) \cdots k \cdots (2j-1)$ for~$j \in [n]$ and~$k > 2j$.
\end{itemize}
\end{definition}

\begin{example}
\label{exm:allSmallWigglyPermutations}
For~$n = 1$, the $2$ wiggly permutations of~$[2]$ are~$12$ and~$21$.
For~$n = 2$, the $14$ wiggly permutations of~$[4]$ are
\[
1234, 1243, 1342, 1423, 1432, 2134, 2143, 3412, 3421, 4123, 4132, 4213, 4312, 4321.
\]
\end{example}

\begin{example}
\label{exm:specialWigglyPermutations}
The permutations
\[
\sigma_\downarrow \eqdef 1 2 \cdots (2n),
\quad
\sigma_\uparrow \eqdef (2n) \cdots 2 1,
\quad
\sigma_\leftarrow \eqdef (2n) \cdots 2 1 \cdots (2n-1),
\quad
\sigma_\rightarrow \eqdef 1 \cdots (2n-1) (2n) \cdots 2
\]
are wiggly permutations of~$[2n]$.
\end{example}

\begin{remark}
\label{rem:numberWigglyPermutations}
The number~$wp_n$ of wiggly permutations of~$[2n]$ is given by
\[
\begin{array}{c|ccccccccc}
n & 1 & 2 & 3 & 4 & 5 & 6 & 7 & 8 & \dots \\
\hline
wp_n & 2 & 14 & 176 & 3232 & 78384 & 2366248 & 85534176 & 3602770400 & \dots
\end{array}
\]
To compute~$wp_n$, denote by $wp_n(x)$ the polynomial where the coefficient of~$x^i$ is the number of wiggly permutations of~$[2n]$ with $i$ \defn{admissible gaps} (\ie gaps~$\gamma$ between two consecutive positions such that there is no~$j \in [n]$ such that the value~$2j$ appears before the gap~$\gamma$ while the value~$2j-1$ appears after the gap~$\gamma$).
For instance,
\begin{align*}
	wp_1(x) & = x + 1, \\
	wp_2(x) & = 3 x^3 + 5 x^2 + 4 x + 2, \\
	wp_3(x) & = 15 x^5 + 35 x^4 + 44 x^3 + 40 x^2 + 28 x + 14, \\
	wp_4(x) & = 105 x^7 + 315 x^6 + 520 x^5 + 630 x^4 + 620 x^3 + 514 x^2 + 352 x + 176.
\end{align*}
We invite the reader to check the expressions of~$wp_1(x)$ and~$wp_2(x)$ with \cref{exm:allSmallWigglyPermutations}.
Any wiggly permutation of~$[2n]$ can be obtained from a wiggly permutation of~$[2n-2]$ by inserting the last values~$2n-1$ and~$2n$.
Each wiggly permutation of~$[2n-2]$ with $i$ admissible gaps contributes to:
\begin{itemize}
\item $i+2$ wiggly permutations of~$[2n]$ with $i+2$ admissible gaps and where~$2n-1$ appears before~$2n$ (they must appear consecutively),
\item $j+1$ wiggly permutations of~$[2n]$ with $j$ admissible gaps and where $2n$ appears before~$2n-1$, for all~$0 \le j \le i+1$.
\end{itemize}
For instance, the wiggly permutation~$78621354$ of~$[8]$ has admissible gaps~$\bullet 7 \bullet 8 \bullet 62135 \bullet 4 \bullet$ and thus gives rise to the wiggly permutations
\[
9\mathrm{X}78621354, \quad 79\mathrm{X}8621354, \quad 789\mathrm{X}621354, \quad 78621359\mathrm{X}4, \quad 786213549\mathrm{X}
\]
of~$[10]$ where the $9$ appears before~$10$, and the wiggly permutations
\begin{align*}
& \mathrm{X}978621354, \quad \mathrm{X}798621354, \quad \mathrm{X}789621354, \quad \mathrm{X}786213594, \quad \mathrm{X}786213549, \\
& 7\mathrm{X}98621354, \quad 7\mathrm{X}89621354, \quad 7\mathrm{X}86213594, \quad 7\mathrm{X}86213549, \quad 78\mathrm{X}9621354, \\
& 78\mathrm{X}6213594, \quad 78\mathrm{X}6213549, \quad 7862135\mathrm{X}94, \quad 7862135\mathrm{X}49, \quad 78621354\mathrm{X}9
\end{align*}
of~$[10]$ where the $10$ appears before~$9$ (for brievety, we have replaced $10$ by $\mathrm{X}$ in the permutations).
We thus obtain the same recursive equation on~$wp_n(x)$ as in \cref{rem:numberWigglyPseudotriangulations}.
\end{remark}

Recall that the \defn{inversion set} and \defn{non-inversion set} of a permutation~$\sigma$ of~$[2n]$ are the sets
\begin{align*}
\inv(\sigma) & \eqdef \set{\big( \sigma(i), \sigma(j) \big)}{1 \le i < j \le 2n \text{ and } \sigma(i) > \sigma(j)}, \\
\ninv(\sigma) & \eqdef \set{\big( \sigma(i), \sigma(j) \big)}{1 \le i < j \le 2n \text{ and } \sigma(i) < \sigma(j)}.
\end{align*}
Note that~$\inv(\sigma)$ and~$\ninv(\sigma)$ are transitive: $\{(c,b), (b,a)\} \subseteq \inv(\sigma)$ implies~$(c,a) \in \inv(\sigma)$ for all~${1 \le a < b < c \le 2n}$, and similarly for~$\ninv$.
The inversion sets of wiggly permutations are characterized as follows.

\begin{lemma}
\label{lem:inversionSetsWigglyPermutations}
A permutation~$\sigma$ of~$[2n]$ is a wiggly permutation if and only if for all~$j \in [n]$,
\begin{itemize}
\item $(2j-1, i) \in \inv(\sigma)$ implies $(2j, i) \in \inv(\sigma)$ for all~$i < 2j-1$, and
\item $(k, 2j-1) \in \inv(\sigma)$ implies $(k, 2j) \in \inv(\sigma)$ for all~$k > 2j$.
\end{itemize}
A similar characterization holds for~$\ninv(\sigma)$ by exchanging~$2j-1$ and~$2j$.
\end{lemma}

\begin{proof}
Immediate as it corresponds to the pattern avoidance description of \cref{def:wigglyPermutation}.
\end{proof}

The (left) \defn{weak order} on permutations of~$[2n]$ is defined as the inclusion order of their inversion sets (or equivalently, the reverse inclusion of their non-inversion sets).
Its cover relations are given by the exchanges of two entries at consecutive positions.
It is a (congruence uniform) lattice, and the join and meet satisfy
\[
\inv(\sigma \join \tau) = \big( \inv(\sigma) \cup \inv(\tau) \big)^\textrm{tc}
\qquad\text{and}\qquad
\ninv(\sigma \meet \tau) = \big( \ninv(\sigma) \cup \ninv(\tau) \big)^\textrm{tc}
\]
where~$X^\mathrm{tc}$ denotes the transitive closure of~$X$.

The following statements are illustrated on \cref{fig:wigglyLattice}\,(right) when~${n = 2}$.
\begin{figure}
\centerline{\includegraphics[scale=1.1]{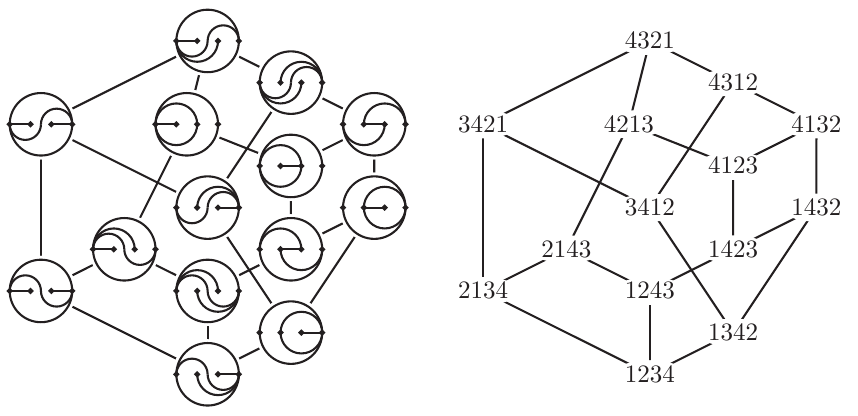}}
\caption{The wiggly lattice~$\wigglyLattice_2$ on wiggly pseudotriangulations (left) and on wiggly permutations (right).}
\label{fig:wigglyLattice}
\end{figure}

\begin{proposition}
The wiggly permutations induce a sublattice of the weak order on permutations of~$[2n]$, that we call the \defn{wiggly lattice}~$\wigglyLattice_n$.
\end{proposition}

\begin{proof}
Consider two wiggly permutations~$\rho$ and~$\sigma$ of~$[2n]$, and let~$\tau = \rho \join \sigma$.
Assume that there is~$1 \le i < 2j-1 < 2n$ such that~$(2j-1, i) \in \inv(\tau)$.
As $\inv(\tau) = \big( \inv(\rho) \cup \inv(\sigma) \big)^\textrm{tc}$, there exist~$i \le i' < 2j-1$ such that~$(i', i) \in \inv(\tau)$ and~$(2j-1, i') \in \inv(\rho) \cup \inv(\sigma)$.
Since~$\rho$ and~$\sigma$ are wiggly permutations, we obtain by \cref{lem:inversionSetsWigglyPermutations} that~$(2j, i') \in \inv(\rho) \cup \inv(\sigma) \subseteq \inv(\tau)$.
As~$(i', i) \in \inv(\tau)$ and~$(2j, i') \in \inv(\tau)$ and~$\inv(\tau)$ is transitive, we conclude that~$(2j,i) \in \inv(\tau)$.
A similar argument shows that~$(k, 2j-1) \in \inv(\tau)$ implies~$(k, 2j) \in \inv(\tau)$ for all~$1 < 2j < k \le 2n$.
By \cref{lem:inversionSetsWigglyPermutations}, we conclude that~$\tau = \rho \join \sigma$ is a wiggly permutation.
The proof is similar for~$\rho \meet \sigma$, using $\ninv$ instead of~$\inv$.
\end{proof}

We now aim at describing the cover relations of the wiggly lattice.
Recall that an \defn{ascent} (resp.~\defn{descent}) in a permutation~$\sigma$ is a position~$j$ such that~$\sigma(j) < \sigma(j+1)$ (resp.~$\sigma(j) > \sigma(j+1)$).

\begin{lemma}\label{lem:wigglyCoverRelation}
Consider an ascent~$j$ of a wiggly permutation~$\sigma$ of~$[2n]$.
Let~$i \eqdef \min \big( j, \sigma^{-1} \big( \sigma(j)+1 \big) \big)$ if~$\sigma(j)$ is odd, and $i \eqdef j$ otherwise.
Let~$k \eqdef \max \big( j+1, \sigma^{-1} \big( \sigma(j+1)+1 \big) \big)$ if~$\sigma(j+1)$ is odd, and $k \eqdef j+1$ otherwise.
Then
\[
\sigma^j \eqdef \sigma(1) \dots \sigma(i-1) \sigma(j+1) \dots \sigma(k) \sigma(i) \dots \sigma(j) \sigma(k+1) \dots \sigma(2n).
\]
is the minimal wiggly permutation such that~$\inv(\sigma) \cup \big\{ \big( \sigma(j+1), \sigma(j) \big) \big\} \subseteq \inv(\sigma^j)$.
\end{lemma}

\begin{proof}
We start with a few observations.
Let~$i < p \le j$ (resp.~$j+1 \le q < k$).
As~$\sigma$ is a wiggly permutation, we have~$\sigma(p) < \sigma(i)$ (resp.~$\sigma(q) > \sigma(k)$).
Hence~$\big( \sigma(i), \sigma(p) \big)$ (resp.~$\big( \sigma(q), \sigma(k) \big)$) belongs to~$\inv(\sigma)$.
Moreover, since~$j$ is an ascent of~$\sigma$, we obtain that~$\sigma(p) < \sigma(q)$ for all~${i \le p \le j < q \le k}$.
Hence, $\inv(\sigma^j) = \inv(\sigma) \cup \bigset{ \big( \sigma(q), \sigma(p) \big) }{i \le p \le j < q \le k}$.

We now prove that~$\sigma^j$ is indeed a wiggly permutation of~$[2n]$.
We consider~$v \in [n]$ and~$u < 2v-1$ such that~$(2v-1, u) \in \inv(\sigma^j)$, and we show that~$(2v, u) \in \inv(\sigma^j)$.
If~$(2v, u) \in \inv(\sigma)$, we are done as~$\inv(\sigma) \subset \inv(\sigma^j)$.
If~$(2v, u) \not\in \inv(\sigma)$, then~$(2v-1, u) \not\in \inv(\sigma)$ by \cref{lem:inversionSetsWigglyPermutations}, and~$i \le \sigma^{-1}(u) \le j < \sigma^{-1}(2v-1) \le k$ since~${\inv(\sigma^j) \ssm \inv(\sigma) = \bigset{ \big( \sigma(q), \sigma(p) \big) }{i \le p \le j < q \le k}}$.
We have~$\sigma^{-1}(2v) \le k$ since~$\sigma$ is a wiggly permutation.
As~$(2v, u) \not\in \inv(\sigma)$ and~$\sigma(p) < \sigma(q)$ for all~$i \le p \le j < q \le k$, we conclude that~$j < \sigma^{-1}(2v) < k$.
Hence, we obtain that ${i \le \sigma^{-1}(u) \le j < \sigma^{-1}(2v) \le k}$, so that~$(2v,u) \in \inv(\sigma^j)$.
Similarly~${(w, 2v-1) \in \inv(\sigma^j)}$ implies~$(w, 2v) \in \inv(\sigma^j)$ for all~$v \in [n]$ and~$w > 2v$.
We conclude that~$\sigma^j$ is indeed a wiggly permutation of~$[2n]$ by \cref{lem:inversionSetsWigglyPermutations}.

Consider now any wiggly transposition~$\tau$ such that~$\inv(\sigma) \cup \big\{ \big( \sigma(j+1), \sigma(j) \big) \big\} \subseteq \inv(\tau)$.
By \cref{lem:inversionSetsWigglyPermutations}, we obtain that~$\inv(\tau)$ also contains 
$\big( \sigma(k), \sigma(i) \big)$.
Since~$\inv(\sigma) \subset \inv(\tau)$ and~$\inv(\sigma)$ contains~$\big( \sigma(i), \sigma(p) \big)$ and~$\big( \sigma(q), \sigma(k) \big)$ for all~$i < p \le j$ and~$j+1 \le q < k$, we conclude by transitivity that~$\inv(\tau)$ contains~$\big( \sigma(q), \sigma(p) \big)$ for all~$i \le p \le j < q \le k$.
\end{proof}

\begin{proposition}
Each wiggly permutation~$\sigma$ covers (resp.~is covered by) as many wiggly permutations as its number of descents (resp.~ascents).
Hence, the cover graph of~$\wigglyLattice_n$ is regular of degree~$2n-1$ and connected.
\end{proposition}

\begin{proof}
Observe that~$\sigma^j \ne \sigma^{j'}$ for distinct ascents~$j \ne j'$.
As any permutation~$\sigma'$ larger than~$\sigma$ satisfies~${\inv(\sigma') \supseteq \inv(\sigma) \cup \{\sigma(j+1), \sigma(j)\}}$ for some ascent~$j$ of~$\sigma$, we thus obtain that~$j \to \sigma^j$ is a bijection from the ascents of~$\sigma$ to the wiggly permutations covering~$\sigma$.
The proof is symmetrical for the bijection from descents of~$\sigma$ to wiggly permutations covered by~$\sigma$.
\end{proof}

\begin{remark}
\label{rem:wigglyLatticeAntiIsomorphism}
Let~$\varepsilon$ be the permutation of~$[2n]$ exchanging~$2j-1$ with~$2j$ for all~$j \in [n]$.
Then the maps~$\sigma_1 \dots \sigma_{2n} \mapsto (2n+1-\sigma_1) \dots (2n+1-\sigma_{2n})$ and~$\sigma_1 \dots \sigma_{2n} \mapsto \varepsilon(\sigma_{2n}) \dots \varepsilon(\sigma_1)$ are two obvious bijections on wiggly permutations of~$[2n]$. Moreover, the first map is an anti-isomorphism of the wiggly lattice~$\wigglyLattice_n$ (while the second map is not).
\end{remark}


\subsection{Bijection}
\label{subsec:bijection}

We now prove that wiggly pseudotriangulations and wiggly permutations are in bijection, and that this bijection induces an isomorphism from the wiggly flip graph to the cover graph of the wiggly lattice.
The following two definitions are illustrated in \cref{fig:wigglyLattice,fig:bijection}.

\begin{figure}
\centerline{\includegraphics[scale=1.7]{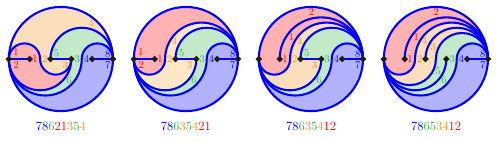}}
\caption{The bijection between wiggly pseudotriangulations and wiggly permutations.}
\label{fig:bijection}
\end{figure}

\begin{definition}
\label{def:bijection1}
Consider a wiggly pseudotriangulation~$T$.
Following each wiggly pseudotriangle of~$T$ in counterclockwise direction, we label by~$2h-1$ (resp.~$2h$) the corner immediately preceding (resp.~following) the hinge~$h$ (the remaining corner remains unlabelled).
We define~$\Phi(T)$ as the permutation obtained by reading these labels from bottom to top, meaning that for each internal wiggly arc~$\alpha$ of~$T$, the label incident to~$\alpha$ and below~$\alpha$ appears just before the label incident to~$\alpha$ and above~$\alpha$.
\end{definition}

\begin{example}
The map~$\Phi$ sends the wiggly pseudotriangulations of \cref{exm:specialWigglyPseudotriangulations} to the wiggly permutations of \cref{exm:specialWigglyPermutations}.
Namely, $\Phi(T_\downarrow) = \sigma_\downarrow$, $\Phi(T_\uparrow) = \sigma_\uparrow$, $\Phi(T_\leftarrow) = \sigma_\leftarrow$, and~$\Phi(T_\rightarrow) = \sigma_\rightarrow$.
\end{example}

\begin{definition}
\label{def:bijection2}
For a wiggly permutation~$\sigma$ of~$[2n]$, we define~$\Psi(\sigma) \eqdef \set{\alpha(\sigma, k)}{k \in [2n-1]}$, where for~$k \in [2n-1]$, we have~$\alpha(\sigma, k) \eqdef ( i, j, A, B)$ with
\begin{align*}
i & \eqdef \max \{0\} \cup \set{i \in [n]}{2i-1 \in \sigma([k]) \not\ni 2i}, \\
j & \eqdef \min \{n+1\} \cup \set{j \in [n]}{2j \in \sigma([k]) \not\ni 2j-1}, \\
A & \eqdef \set{\ell \in {]i,j[}}{\{2\ell-1, 2\ell\} \subseteq \sigma([k])}, \\
B & \eqdef \set{\ell \in {]i,j[}}{\{2\ell-1, 2\ell\} \cap \sigma([k]) = \varnothing}.
\end{align*}
\end{definition}

\begin{example}
The map~$\Psi$ sends the wiggly permutations of \cref{exm:specialWigglyPermutations} to the wiggly pseudotriangulations of \cref{exm:specialWigglyPseudotriangulations}.
Namely, $\Psi(\sigma_\downarrow) = T_\downarrow$, $\Psi(\sigma_\uparrow) = T_\uparrow$, $\Psi(\sigma_\leftarrow) = T_\leftarrow$, and~$\Psi(\sigma_\rightarrow) = T_\rightarrow$.
\end{example}

\begin{proposition}
\label{prop:bijection}
The maps~$\Phi$ of \cref{def:bijection1} and $\Psi$ of \cref{def:bijection2} define inverse bijections between the wiggly pseudotriangulations and the wiggly permutations, which induce a directed graph isomorphism between the wiggly increasing flip graph~$\wigglyIncreasingFlipGraph_n$ and the Hasse diagram of the wiggly lattice~$\wigglyLattice_n$.
\end{proposition}

\begin{proof}
First, we show that \(\Phi(T)\) is a well-defined wiggly permutation for each wiggly pseudotriangulation \(T\).
As every~$j \in [n]$ is the hinge of precisely one pseudotriangle of~$T$, we obtain that~\(\Phi(T)\) is by construction a permutation~$\sigma$ of~$[2n]$.
For any \(j \in [n]\), let \(t_j\) be the pseudotriangle of \(T\) that has hinge \(j\).
Suppose that \(2j-1\) appears before \(2j\) in \(\sigma\) for some fixed \(j \in [n]\).
This means that the pseudotriangle~\(t_j\) has its unlabelled corner to the left of its labelled corners.
In other words, its hinge opens to the right (see \eg $t_2$ in~\cref{fig:bijection}).
Consider any label \(i\) that we encounter as we traverse from \(2j-1\) to \(2j\).
This label belongs to some pseudotriangle \(t_h\) with~$h > j$.
In other words, \(\sigma\) avoids the pattern \(2j-1 \cdots i \cdots 2j\) for \(i < 2j-1\).
The other case (in which \(2j\) appears before \(2j-1\) in \(\sigma\)) is symmetric.

Next, we show that \(\Psi(\sigma)\) is a wiggly pseudotriangulation for each wiggly permutation \(\sigma\).
Observe first that for any~$1 \le u < v \le n$, we cannot see~$\{2u, 2v-1\}$ before~$\{2u-1, 2v\}$ in~$\sigma$ as it contains no forbidden pattern of \cref{def:wigglyPermutation}.
Using the notations of \cref{def:bijection2}, this implies that~$i < j$ and that~${]i,j[} = A \sqcup B$, so that~$\alpha(\sigma,k) \eqdef (i, j, A, B)$ is indeed a wiggly arc for each~$k \in [2n-1]$.
By~\cref{prop:wigglyPseudotriangulations}, it thus suffices to check that for each \(k \neq k'\) in \([2n-1]\), the wiggly arcs 
\(\alpha(\sigma,k)\) and~\(\alpha(\sigma,k')\) are distinct and compatible.
\begin{enumerate}
\item \textbf{Distinctness.}
Suppose by contradiction that \(\alpha(\sigma,k) = (i,j,A,B) = \alpha(\sigma,k')\) for some~\({k < k'}\).
Let~${\ell \in \sigma([k'] \ssm [k])}$ and~$\bar\ell \eqdef \lceil \ell/2 \rceil$.
As~$\ell \notin \sigma([k])$, we have~$\bar\ell \notin A$.
As~$\ell \in \sigma([k'])$, we have~$\bar\ell \notin B$.
As~$\{2i-1, 2j\} \subseteq \sigma([k])$ and~$\{2i, 2j-1\} \cap \sigma([k']) = \varnothing$ and~$\sigma$ avoids the forbidden pattern of \cref{def:wigglyPermutation}, we have~$2i < \ell < 2j-1$, so that~$i < \bar\ell < j$.
This contradicts the fact that~${]i,j[} = A \sqcup B$.

\item \textbf{Compatibility.}
Let $\alpha(\sigma, k) = (i,j,A,B)$ and $\alpha(\sigma, k') = (i',j',A',B')$ for some \(k < k'\).
It is clear from the definition that \(i \neq j'\), as the sets over which we take the $\max$ and $\min$ to define \(i\) and \(j'\) respectively are disjoint.
Similarly, \(j \neq i'\).
Hence, the two wiggly arcs are pointed.
To check that the wiggly arcs are non-crossing, we show that~$(A \cap B') \cup (\{i,j\} \cap B') \cup (A \cap \{i',j'\}) = \varnothing$.
Indeed~$A \cap B' = \varnothing$ since~$A \subseteq A'$ because~$\sigma([k]) \subseteq \sigma([k'])$.
Moreover, \(\{i,j\} \cap B' = \emptyset\) since~\(\{2i - 1, 2j\} \subseteq \sigma([k]) \subseteq \sigma([k'])\).
Finally, \(A \cap \{i',j'\} = \emptyset\) since~\(\{2i', 2j'-1\} \cap \sigma([k]) \subseteq \{2i', 2j'-1\} \cap \sigma([k']) = \varnothing\).
\end{enumerate}

Next, we show that the maps \(\Phi\) and \(\Psi\) are inverse bijections.
Let \(\sigma\) be a wiggly permutation, so that \(\Psi(\sigma)\) is a wiggly pseudotriangulation.
Let~$k \in [2n-1]$ and~$m \eqdef \lceil \sigma(k)/2 \rceil$.
Then observe~that
\begin{itemize}
\item if $\{2m, 2m-1\} \subseteq \sigma([k])$, then~$k$ is minimal such that~$\alpha(\sigma,k)$ passes above the point~$m$,
\item if~$\sigma(k) = 2m$ and~$2m-1 \notin \sigma([k])$, then~$k$ is minimal such that~$\alpha(\sigma,k)$ has right endpoint~$m$,
\item if~$\sigma(k) = 2m-1$ and~$2m \notin \sigma([k])$, then~$k$ is minimal such that~$\alpha(\sigma,k)$ has left endpoint~$m$.
\end{itemize}
It follows that in the wiggly pseudotriangulation~\(\Psi(\sigma)\), the wiggly arc \(\alpha(\sigma, k)\) is precisely the wiggly arc lying immediately above the angle labelled by~\(\sigma(k)\).
Therefore, we have~\(\Phi(\Psi(\sigma)) = \sigma\).

Conversely, let \(T\) be a wiggly pseudotriangulation, so that \(\sigma = \Phi(T)\) is a wiggly permutation.
Then for any \(k \in [2n-1]\), the wiggly arc immediately above the angle of~$T$ labelled by~$\sigma(k)$ is precisely the one constructed as~\(\alpha(\sigma, k)\).
This shows that \(\Psi(\Phi(T)) = T\).

\pagebreak
Finally, consider a cover relation~$\sigma \lessdot \sigma^j$ in the wiggly lattice~$\wigglyLattice_n$ as described in~\cref{lem:wigglyCoverRelation}.
Using the notations of~\cref{lem:wigglyCoverRelation}, observe that
\begin{itemize}
\item $\alpha(\sigma^j, \ell) = \alpha(\sigma, \ell)$ for all~$\ell < i$ or~$\ell > k$,
\item $\alpha(\sigma^j, \ell) = \alpha(\sigma, \ell+k-j)$ for all~$i \le \ell < j$, and
\item $\alpha(\sigma^j, \ell) = \alpha(\sigma, \ell-j+i-1)$ for all~$j < \ell \le k$.
\end{itemize}
Hence, $\Psi(\sigma) \ssm \{\alpha(\sigma, j)\} = \Psi(\sigma^j) \ssm \{\alpha(\sigma^j, j)\}$.
This implies that~$\Psi$ is a directed graph morphism from the Hasse diagram of the wiggly lattice~$\wigglyLattice_n$ to the wiggly increasing flip graph~$\wigglyIncreasingFlipGraph_n$.
As the underlying undirected graphs are both regular of degree~$2n-1$, we conclude that~$\Phi$ and~$\Psi$ induce directed graph isomorphisms between the wiggly increasing flip graph~$\wigglyIncreasingFlipGraph_n$ and the Hasse diagram of the wiggly lattice~$\wigglyLattice_n$.
\end{proof}

\begin{remark}
Transporting the wiggly lattice~$\wigglyLattice_n$ via the map~$\Psi$ to the wiggly pseudotriangulation, note that the minimal wiggly pseudotriangulation is~$\Psi(\sigma_\downarrow) = T_\downarrow$ while the maximal wiggly pseudotriangulation is~$\Psi(\sigma_\uparrow) = T_\uparrow$. See \cref{exm:specialWigglyPseudotriangulations,fig:specialWigglyPseudotriangulations} for illustrations.
\end{remark}

\begin{remark}
The map~$\Phi$ sends the wiggly complex automorphisms of \cref{rem:wigglyComplexAutomorphism} to the bijections on wiggly permutations of \cref{rem:wigglyLatticeAntiIsomorphism}. In particular, it sends the vertical reflection on wiggly pseudotriangulations to the wiggly lattice anti-isomorphism on wiggly permutations given by~$\sigma_1 \dots \sigma_{2n} \mapsto (2n+1-\sigma_1) \dots (2n+1-\sigma_{2n})$.
\end{remark}

\begin{remark}
The map~$\Phi$ sends wiggly pseudotriangulations with $i$ final internal wiggly arcs to wiggly permutations with $i$ admissible gaps, explaining the similarities of \cref{rem:numberWigglyPseudotriangulations,rem:numberWigglyPermutations}.
For instance, the wiggly pseudotriangulations of \cref{fig:pseudotriangulationsCounting} illustrating the description of \cref{rem:numberWigglyPseudotriangulations} correspond to the wiggly permutations
\[
\begin{array}{cccc}
78621354 & 9\mathrm{X}78621354 & 789\mathrm{X}621354 & 78621359\mathrm{X}4 \\
\mathrm{X}786213594 & 7\mathrm{X}89621354 & 78\mathrm{X}6213549 & 7862135\mathrm{X}49
\end{array}
\]
illustrating the description of \cref{rem:numberWigglyPermutations}.
\end{remark}


\subsection{Graph properties}
\label{subsec:graphProperties}

We conclude this combinatorial section with two open questions and a conjecture on graphical properties of the wiggly flip graph~$\wigglyFlipGraph_n$ (or equivalently of the cover graph of the wiggly lattice~$\wigglyLattice_n$), motivated by the relevance of similar questions for the associahedron.

First, we are interested in the graph diameter~$\delta_n$ of~$\wigglyFlipGraph_n$.
It is clearly bounded below by~$2n-1$ (as~$T_\leftarrow \cap T_\rightarrow = \{\alpha_\mathrm{top}, \alpha_\mathrm{bot}\}$, all $2n-1$ internal wiggly arcs must sometimes be flipped at least once) and above by~$\binom{2n}{2}$ (as we have a sublattice of the weak order on permutations of~$[2n]$).
The first few values of~$\delta_n$ are
\[
\begin{array}{c|ccccccccc}
n & 1 & 2 & 3 & 4 & 5 & \dots \\
\hline
\delta_n & 1 & 4 & 8 & 14 & 20 & \dots
\end{array}
\]
but they might be misleading as for the associahedron~\cite{Pournin}.
This raises the following question.

\begin{question}
Compute (or estimate) the diameter~$\delta_n$ of the wiggly flip graph~$\wigglyFlipGraph_n$.
\end{question}

Related to the estimation of the diameter~$\delta_n$ is the study of the properties of the geodesic paths in~$\wigglyFlipGraph_n$.
We refer to~\cite{CeballosPilaud-diameterDAssociahedron} for a discussion on geodesic properties in a polytope, in connection to the revisiting path property and the complexity of the simplex algorithm~\cite{Santos-surveyHirsch,Santos-Hirsch}.

\begin{question}
Does the wiggly complex has the non-leaving face property of~\cite{CeballosPilaud-diameterDAssociahedron}?
\end{question}

Finally, we checked the following conjecture computationally up to~$n = 3$.

\begin{conjecture}
\label{conj:Hamiltonian}
The wiggly flip graph~$\wigglyFlipGraph_n$ admits an Hamiltonian path (or even cycle).
\end{conjecture}

Note that unfortunately, wiggly permutations do not form a zigzag language in the sense of~\cite{HartungHoangMutzeWilliams} (which would have guarantied the existence of a Gray code).


\pagebreak
\section{Wiggly fan and wigglyhedron}
\label{sec:geometry}

In this section, we use $\b{g}$- and $\b{c}$-vectors (\cref{subsec:gcvectors}) to define the wiggly fan (\cref{subsec:wigglyFan}) and the wigglyhedron (\cref{subsec:wigglyhedron}), which provide geometric realizations of the wiggly complex and of the wiggly lattice.


\subsection{Polyhedral geometry}
\label{subsec:polyhedralGeometry}

We refer to \cite{Ziegler-polytopes} for a reference on polyhedral geometry, and only remind the basic notions needed later in the paper.

A (polyhedral) \defn{cone} is the positive span~$\R_{\ge 0}\b{R}$ of a finite set~$\b{R}$ of vectors of~$\R^d$ or equivalently, the intersection of finitely many closed linear half-spaces of~$\R^d.$ 
The \defn{faces} of a cone are its intersections with its supporting hyperplanes. 
The \defn{rays} (resp.~\defn{facets}) are the faces of dimension~$1$ (resp.~ codimension~$1$).
A cone is \defn{simplicial} if its rays are linearly independent.
A (polyhedral) \defn{fan}~$\Fan$ is a set of cones such that any face of a cone of~$\Fan$ belongs to~$\Fan$, and any two cones of~$\Fan$ intersect along a face of both. 
A fan is \defn{essential} if the intersection of its cones is the origin, \defn{complete} if the union of its cones covers~$\R^d$, and \defn{simplicial} if all its cones are simplicial.

Note that a simplicial fan defines a simplicial complex on its rays (the simplices of the simplicial complex are the subsets of rays which span a cone of the fan).
Conversely, given a simplicial complex~$\Delta$ with ground set~$V$, one can try to realize it geometrically by associating a ray~$\b{r}_v$ of~$\R^d$ to each~$v \in V$, and the cone~$\R_{\ge 0}\b{R}_\triangle$ generated by the set~$\b{R}_\triangle \eqdef \set{\b{r}_v}{v \in \triangle}$ to each~$\triangle \in \Delta$.
To show that the resulting cones indeed form a fan, we will need the following statement, which can be seen as a reformulation of~\cite[Coro.~4.5.20]{DeLoeraRambauSantos}.

\begin{proposition}
\label{prop:characterizationFan}
Consider a simplicial $(d-1)$-dimensional pseudomanifold~$\Delta$ without boundary on a ground set~$V$ and a set of vectors~$(\b{r}_v)_{v \in V}$ of~$\R^d$, and define~$\b{R}_\triangle \eqdef \set{\b{r}_v}{v \in \triangle}$ for any~$\triangle \in \Delta$.
Then the collection of cones~$\set{\R_{\ge 0}\b{R}_\triangle}{\triangle \in \Delta}$ forms a complete simplicial fan of~$\R^d$ if and~only~if
\begin{itemize}
\item there exists a vector~$\b{v}$ of~$\R^d$ contained in only one of the cones~$\R_{\ge 0}\b{R}_\triangle$ for~$\triangle \in \Delta$,
\item for any two adjacent facets~$\triangle, \triangle'$ of~$\Delta$ with~$\triangle \ssm \{v\} = \triangle' \ssm \{v'\}$, we have~$\lambda_v \lambda_{v'} > 0$~where
\[
\lambda_v \, \b{r}_v + \lambda_{v'} \, \b{r}_{v'} + \sum_{w \in \triangle \cap \triangle'} \lambda_w \, \b{r}_w = 0
\]
denotes the unique (up to rescaling) linear dependence on~$\b{R}_{\triangle \cup \triangle'}$.
\end{itemize}
\end{proposition}

A \defn{polytope} is the convex hull of finitely many points of~$\R^d$ or equivalently, a bounded intersection of finitely many closed affine half-spaces of~$\R^d$.
The \defn{faces} of a polytope are its intersections with its supporting hyperplanes.
The \defn{vertices} (resp.~\defn{edges}, resp.~\defn{facets}) are the faces of dimension~$0$ (resp.~dimension~$1$, resp.~codimension~$1$).

The \defn{normal cone} of a face~$\polytope{F}$ of a polytope~$\polytope{P}$ is the cone generated by the normal vectors to the supporting hyperplanes of~$\polytope{P}$ containing~$\polytope{F}$.
Said differently, it is the cone of vectors~$\b{c}$ of~$\R^d$ such that the linear form~$\b{x} \mapsto \dotprod{\b{c}}{\b{x}}$ on~$\polytope{P}$ is maximized by all points of the face~$\polytope{F}$.
The \defn{normal fan} of~$\polytope{P}$ is the set of normal cones of all its faces.

\enlargethispage{.6cm}
Consider now a complete simplicial fan~$\Fan$ of~$\R^d$ with rays~$(\b{r}_v)_{v \in V}$ and cones~$\R_{\ge 0} \b{R}_\triangle$ for~${\triangle \in \Delta}$, where~$\b{R}_\triangle \eqdef \set{\b{r}_v}{v \in \triangle}$ as in \cref{prop:characterizationFan}.
To realize the fan~$\Fan$, one can try to pick a height vector~$\b{h} \eqdef (h_v)_{v \in V} \in \R^V$ and consider the polytope
\(
\polytope{P}_{\b{h}} \eqdef \set{\b{x} \in \R^d}{\dotprod{\b{r}_v}{\b{x}} \le h_v \text{ for all } v \in V}.
\)
The following classical statement characterizes the height vectors~$\b{h}$ for which the fan~$\Fan$ is the normal fan of this polytope~$\polytope{P}_{\b{h}}$.
We borrow the formulation from~\cite[Lem.~2.1]{ChapotonFominZelevinsky}.

\begin{proposition}
\label{prop:characterizationPolytopalFan}
Let~$\Fan$ be an essential complete simplicial fan in~$\R^n$ with rays~$(\b{r}_v)_{v \in V}$ and cones~$\R_{\ge 0} \b{R}_\triangle$ for~$\triangle \in \Delta$.
Then the following are equivalent for any height vector~$\b{h} \in \R^V$:
\begin{itemize}
\item The fan~$\Fan$ is the normal fan of the polytope~$\polytope{P}_{\b{h}} \eqdef \set{\b{x} \in \R^d}{\dotprod{\b{r}_v}{\b{x}} \le h_v \text{ for all } v \in V}$.
\item For two adjacent facets~$\triangle, \triangle'$ of~$\Delta$ with~$\triangle \ssm \{v\} = \triangle' \ssm \{v'\}$, the height vector~$\b{h}$ satisfies the \defn{wall crossing inequality}
\[
\lambda_v \, h_v + \lambda_{v'} \, h_{v'} + \sum_{w \in \triangle \cap \triangle'} \lambda_w \, h_w > 0,
\]
where
\[
\lambda_v \, \b{r}_v + \lambda_{v'} \, \b{r}_{v'} + \sum_{w \in \triangle \cap \triangle'} \lambda_w \, \b{r}_w = 0
\]
denotes the unique linear dependence on~$\b{R}_{\triangle \cup \triangle'}$ such that~$\lambda_v + \lambda_{v'} = 2$.
\end{itemize}
\end{proposition}

\pagebreak
We denote by~$(\b{e}_i)_{i \in [d]}$ the standard basis of~$\R^d$.
For~$I \subseteq [d]$, we define~$\one_I \eqdef \sum_{i \in I} \b{e}_i$, and we often shorten~$\one_{[d]}$ by~$\one_d$.
We denote by~$\HH_d$ the hyperplane of~$\R^d$ defined by the equation~$\dotprod{\b{x}}{\one_d} = 0$, and we denote by~$\pi : \R^d \to \HH_d$ the orthogonal projection on~$\HH_d$, that is~$\pi(\b{x}) \eqdef \b{x} - (\dotprod{\b{x}}{\one_d} / d) \one_d$.


\subsection{$\b{g}$-vectors and $\b{c}$-vectors}
\label{subsec:gcvectors}

We first define the $\b{g}$-vectors, illustrated in \cref{fig:ggMatrices}.
\begin{figure}
\centerline{\raisebox{-1.8cm}{\includegraphics[scale=2]{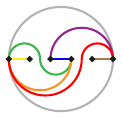}} \quad \input{figures/ggmatrices}}
\caption{The $\hat{\b{g}}$-matrix~$\hat{\b{g}}(T)$ and $\b{g}$-matrix~$\b{g}(T)$ of a wiggly pseudotriangulation~$T$. The $i$th column of these matrices are the $\hat{\b{g}}$-vector~$\hat{\b{g}}(\alpha)$ and the $\b{g}$-vector~$\b{g}(\alpha)$ of the $i$th wiggly arc~$\alpha$ of~$T$ (ordered from bottom to top). The colors of the columns match the colors of the wiggly arcs. The $\b{g}$-vectors are the orthogonal projections of the $\hat{\b{g}}$-vectors on the hyperplane~$\HH_{2n}$.}
\label{fig:ggMatrices}
\end{figure}

\begin{definition}
\label{def:gvectors}
The \defn{$\b{g}$-vector} of a wiggly arc~$\alpha \eqdef (i, j, A, B)$ is the vector~$\b{g}(\alpha)$ of~$\HH_{2n}$ defined as the projection~$\pi \big( \hat{\b{g}}(\alpha) \big)$ of the vector~$\hat{\b{g}}(\alpha) \eqdef \one_{\alpha^+} - \one_{\alpha^-}$ where
\begin{align*}
\alpha^+ & \eqdef \big(\{2i-1, 2j\} \ssm \{-1, 2n+2\} \big) \cup \set{2a-1}{a \in A} \cup \set{2a}{a \in A},
\\
\alpha^- & \eqdef \big(\{2i, 2j-1\} \ssm \{0, 2n+1\} \big) \cup \set{2b-1}{b \in B} \cup \set{2b}{b \in B}.
\end{align*}
If we place the coordinate~$2p-1$ on the left and the coordinate~$2p$ on the right of point~$p$, then~$\hat{\b{g}}(\alpha)$ has a $1$ outside its two endpoints and on both sides of points in~$A$, and a~$-1$ inside its two endpoints and on both sides of points in~$B$.
Note that ${\hat{\b{g}} \big( (0, n+1, [n], \varnothing) \big) = \one_{2n} = - \hat{\b{g}} \big( (0, n+1, \varnothing, [n]) \big)}$, so that~${\b{g} \big( (0, n+1, [n], \varnothing) \big) = \b{0} = \b{g} \big( (0, n+1, \varnothing, [n]) \big)}$.
We define~$\b{g}(D) \eqdef \set{\b{g}(\alpha)}{\alpha \in D^\circ}$ for a set~$D$ of wiggly arcs.
\end{definition}

\begin{example}
\label{exm:specialGMatrices}
\cref{table:specialGMatrices} gathers the $\hat{\b{g}}$-matrices (\ie matrices whose columns are the $\hat{\b{g}}$-vectors) of the wiggly pseudotriangulations of \cref{exm:specialWigglyPseudotriangulations} illustrated in \cref{fig:specialWigglyPseudotriangulations}.

\begin{table}[p]
	\begingroup
	\fontsize{10}{10}\selectfont
	\setlength\arraycolsep{2pt}
	\begin{alignat*}{3}
	\hat{\b{g}}(T_\downarrow) & = \begin{pmatrix}
		1 & 1 & 0 & 1 & \dots & 0 & 1 & 0 \\
		-1 & 1 & 0 & 1 & \dots & 0 & 1 & 0 \\
		-1 & -1 & 1 & 1 & \dots & 0 & 1 & 0 \\
		-1 & -1 & -1 & 1 & \dots & 0 & 1 & 0 \\
		-1 & -1 & -1 & -1 & \dots & 0 & 1 & 0 \\
		\vdots & \vdots & \vdots & \vdots & \ddots & \vdots & \vdots & \vdots \\
		-1 & -1 & -1 & -1 & \dots & 0 & 1 & 0 \\
		-1 & -1 & -1 & -1 & \dots & 1 & 1 & 0 \\
		-1 & -1 & -1 & -1 & \dots & -1 & 1 & 0 \\
		-1 & -1 & -1 & -1 & \dots & -1 & -1 & 1 \\
		-1 & -1 & -1 & -1 & \dots & -1 & -1 & -1 \\
	\end{pmatrix}
	\qquad &
	\hat{\b{g}}(T_\leftarrow) & = \begin{pmatrix}
		-1 & -1 & \dots & -1 & -1 & 1 & \dots & 1 & 1 \\
		-1 & -1 & \dots & -1 & 1 & 1 & \dots & 1 & 1 \\
		-1 & -1 & \dots & -1 & 0 & -1 & \dots & 1 & 1 \\
		-1 & -1 & \dots & 1 & 0 & 1 & \dots & 1 & 1 \\
		-1 & -1 & \dots & 0 & 0 & 0 & \dots & 1 & 1 \\
		\vdots & \vdots & \ddots & \vdots & \vdots & \vdots & \ddots & \vdots & \vdots \\
		-1 & -1 & \dots & 0 & 0 & 0 & \dots & 1 & 1 \\
		-1 & -1 & \dots & 0 & 0 & 0 & \dots & -1 & 1 \\
		-1 & 1 & \dots & 0 & 0 & 0 & \dots & 1 & 1 \\
		-1 & 0 & \dots & 0 & 0 & 0 & \dots & 0 & -1 \\
		1 & 0 & \dots & 0 & 0 & 0 & \dots & 0 & 1 \\
	\end{pmatrix}
	\\[.3cm]
	\hat{\b{g}}(T_\uparrow) & = \begin{pmatrix}
		-1 & -1 & -1 & -1 & \dots & -1 & -1 & -1 \\
		-1 & -1 & -1 & -1 & \dots & -1 & -1 & 1 \\
		-1 & -1 & -1 & -1 & \dots & -1 & 1 & 0 \\
		-1 & -1 & -1 & -1 & \dots & 1 & 1 & 0 \\
		-1 & -1 & -1 & -1 & \dots & 0 & 1 & 0 \\
		\vdots & \vdots & \vdots & \vdots & \ddots & \vdots & \vdots & \vdots \\
		-1 & -1 & -1 & -1 & \dots & 0 & 1 & 0 \\
		-1 & -1 & -1 & 1 & \dots & 0 & 1 & 0 \\
		-1 & -1 & 1 & 1 & \dots & 0 & 1 & 0 \\
		-1 & 1 & 0 & 1 & \dots & 0 & 1 & 0 \\
		1 & 1 & 0 & 1 & \dots & 0 & 1 & 0 \\
	\end{pmatrix}
	\qquad &
	\hat{\b{g}}(T_\rightarrow) & = \begin{pmatrix}
		1 & 0 & \dots & 0 & 0 & 0 & \dots & 0 & 1 \\
		-1 & 0 & \dots & 0 & 0 & 0 & \dots & 0 & -1 \\
		-1 & 1 & \dots & 0 & 0 & 0 & \dots & 1 & 1 \\
		-1 & -1 & \dots & 0 & 0 & 0 & \dots & -1 & 1 \\
		-1 & -1 & \dots & 0 & 0 & 0 & \dots & 1 & 1 \\
		\vdots & \vdots & \ddots & \vdots & \vdots & \vdots  & \ddots & \vdots & \vdots \\
		-1 & -1 & \dots & 0 & 0 & 0 & \dots & 1 & 1 \\
		-1 & -1 & \dots & 1 & 0 & 1 & \dots & 1 & 1 \\
		-1 & -1 & \dots & -1 & 0 & -1 & \dots & 1 & 1 \\
		-1 & -1 & \dots & -1 & 1 & 1 & \dots & 1 & 1 \\
		-1 & -1 & \dots & -1 & -1 & 1 & \dots & 1 & 1 \\
	\end{pmatrix}
	\end{alignat*}
	\endgroup
	\caption{The $\hat{\b{g}}$-matrices of the wiggly pseudotriangulations of \cref{exm:specialWigglyPseudotriangulations}. The $i$th column of~$\hat{\b{g}}(T)$ is the $\hat{\b{g}}$-vector~$\hat{\b{g}}(\alpha)$ of the $i$th wiggly arc~$\alpha$ of~$T$ (ordered from bottom to top). 
}
	\label{table:specialGMatrices}
\end{table}
\end{example}

We will see in \cref{coro:basis} that~$\b{g}(T)$ forms a basis of~$\HH_{2n}$ for any wiggly pseudotriangulation~$T$, which enables us to define its dual basis, illustrated in \cref{fig:gcMatrices}.
\begin{figure}
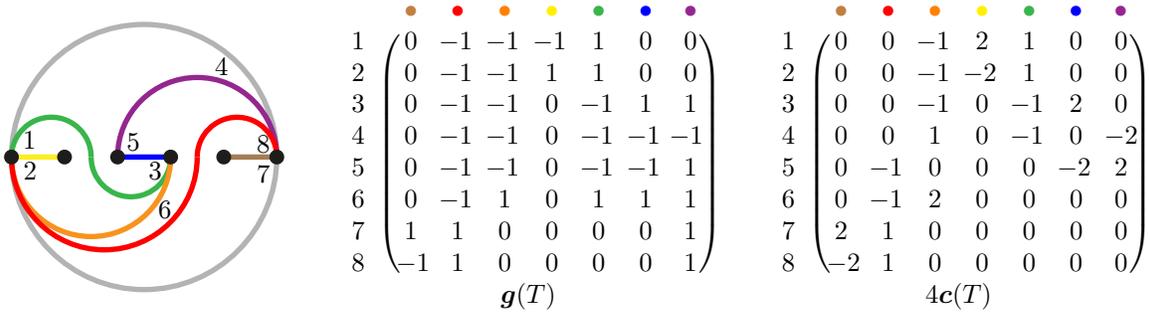

\centerline{\raisebox{-1.8cm}{\includegraphics[scale=2]{pseudotriangulationMatrices}} \quad \input{figures/gcmatrices}}
\caption{The $\b{g}$-matrix~$\b{g}(T)$ and $\b{c}$-matrix~$\b{c}(T)$ of a wiggly pseudotriangulation~$T$. The $i$th column of these matrices are the $\b{g}$-vector~$\b{g}(\alpha)$ and $\b{c}$-vector~$\b{c}(\alpha, T)$ of the $i$th wiggly arc~$\alpha$ of~$T$ (ordered from bottom to top). The colors of the columns match the colors of the wiggly arcs. The $\b{g}$-vectors and $\b{c}$-vectors form dual bases of~$\HH_{2n}$.}
\label{fig:gcMatrices}
\end{figure}

\begin{definition}
\label{def:cvectors}
The \defn{$\b{c}$-vector} of an interior wiggly arc~$\alpha$ in a wiggly pseudotriangulation~$T$ is the vector~$\b{c}(\alpha, T)$ of~$\HH_{2n}$ such that~$\dotprod{\b{c}(\alpha, T)}{\b{g}(\alpha)} = 1$ and~$\dotprod{\b{c}(\alpha, T)}{\b{g}(\alpha')} = 0$ for all interior wiggly arcs~$\alpha' \ne \alpha$ of~$T$.
In other words, $\b{c}(T) \eqdef \set{\b{c}(\alpha, T)}{\alpha \in T^\circ}$ is the dual basis of~$\b{g}(T)$~in~$\HH_{2n}$.
\end{definition}

\begin{remark}
The $\b{c}$-vector~$\b{c}(\alpha, T)$ can also be described combinatorially.
Namely, label the the corners of~$T$ as in \cref{def:bijection1}, and denote by~$u$ and~$v$ the labels incident to~$\alpha$ and respectively below and above~$\alpha$, and let~$\bar u \eqdef \lceil u/2 \rceil$ and~$\bar v \eqdef \lceil v/2 \rceil$.
Then the $\b{c}$-vector~$\b{c}(\alpha, T)$ has coordinates
\begin{itemize}
\item $\b{c}(\alpha, T)_{2w-1} = - \b{c}(\alpha, T)_{2w} = (-1)^{w \in A} (-1)^{u>v}/4$ for all~$w$ strictly between~$\bar u$ and~$\bar v$,
\item $\b{c}(\alpha, T)_u = 1/2$ if~$\bar u = \bar v$, or if~$\bar v$ and the wiggly arcs incident to~$\bar u$ are on the same side of~$\bar u$, and $\b{c}(\alpha, T)_{2\bar u-1} = \b{c}(\alpha, T)_{2\bar u} = 1/4$ otherwise,
\item $\b{c}(\alpha, T)_v = -1/2$ if~$\bar u = \bar v$, or if~$\bar u$ and the wiggly arcs incident to~$\bar v$ are on the same side of~$\bar v$, and $\b{c}(\alpha, T)_{2\bar v-1} = \b{c}(\alpha, T)_{2\bar v} = -1/4$ otherwise,
\item $\b{c}(\alpha, T)_w = 0$ for all other coordinates.
\end{itemize}
The tedious proof of this description is not needed here as we only use the dual basis property.
\end{remark}

\begin{example}
\label{exm:specialCMatrices}
\cref{table:specialCMatrices} gathers the $\b{c}$-matrices (\ie matrices whose columns are the $\b{c}$-vectors) of the wiggly pseudotriangulations of \cref{exm:specialWigglyPseudotriangulations} illustrated in \cref{fig:specialWigglyPseudotriangulations}.

\begin{table}
	\begingroup
	\fontsize{10}{10}\selectfont
	\setlength\arraycolsep{2pt}
	\begin{alignat*}{3}
	4\b{c}(T_\downarrow) & = \begin{pmatrix}
		2 & 0 & 0 & 0 & \dots & 0 & 0 & 0 \\
		-2 & 2 & 0 & 0 & \dots & 0 & 0 & 0 \\
		0 & -1 & 2 & 0 & \dots & 0 & 0 & 0 \\
		0 & -1 & -2 & 2 & \dots & 0 & 0 & 0 \\
		0 & 0 & 0 & -1 & \dots & 0 & 0 & 0 \\
		\vdots & \vdots & \vdots & \vdots & \ddots & \vdots & \vdots & \vdots \\
		0 & 0 & 0 & 0 & \dots & 0 & 0 & 0 \\
		0 & 0 & 0& 0 & \dots & 2 & 0 & 0 \\
		0 & 0 & 0 & 0 & \dots & -2 & 2 & 0 \\
		0 & 0 & 0 & 0 & \dots & 0 & -1 & 2 \\
		0 & 0 & 0 & 0 & \dots & 0 & -1 & -2 \\
	\end{pmatrix}
	\quad &
	4\b{c}(T_\leftarrow) & = \begin{pmatrix}
		0 & 0 & \dots & -1 & -2 & 1 & \dots & 0 & 0 \\
		0 & 0 & \dots & -1 & 2 & 1 & \dots & 0 & 0 \\
		0 & 0 & \dots & 0 & 0 & -2 & \dots & 0 & 0 \\
		0 & 0 & \dots & 2 & 0 & 0 & \dots & 0 & 0 \\
		0 & 0 & \dots & 0 & 0 & 0 & \dots & 0 & 0 \\
		\vdots & \vdots & \ddots & \vdots & \vdots & \vdots & \ddots & \vdots & \vdots \\
		0 & -1 & \dots & 0 & 0 & 0 & \dots & 1 & 0 \\
		-1 & 0 & \dots & 0 & 0 & 0 & \dots & -2 & 1 \\
		-1 & 2 & \dots & 0 & 0 & 0 & \dots & 0 & 1 \\
		0 & 0 & \dots & 0 & 0 & 0 & \dots & 0 & -2 \\
		2 & 0 & \dots & 0 & 0 & 0 & \dots & 0 & 0 \\
	\end{pmatrix}
	\quad
	\\[.3cm]
	4\b{c}(T_\uparrow) & = \begin{pmatrix}
		0 & 0 & 0 & 0 & \dots & 0 & -1 & -2 \\
		0 & 0 & 0 & 0 & \dots & 0 & -1 & 2 \\
		0 & 0 & 0 & 0 & \dots & -2 & 2 & 0 \\
		0 & 0 & 0 & 0 & \dots & 2 & 0 & 0 \\
		0 & 0 & 0 & 0 & \dots & 0 & 0 & 0 \\
		\vdots & \vdots & \vdots & \vdots & \ddots & \vdots & \vdots & \vdots \\
		0 & 0 & 0 & -1 & \dots & 0 & 0 & 0 \\
		0 & -1 & -2 & 2 & \dots & 0 & 0 & 0 \\
		0 & -1 & 2 & 0 & \dots & 0 & 0 & 0 \\
		-2 & 2 & 0 & 0 & \dots & 0 & 0 & 0 \\
		2 & 0 & 0 & 0 & \dots & 0 & 0 & 0 \\
	\end{pmatrix}
	\quad &
	4\b{c}(T_\rightarrow) & = \begin{pmatrix}
		2 & 0 & \dots & 0 & 0 & 0 & \dots & 0 & 0 \\
		0 & 0 & \dots & 0 & 0 & 0 & \dots & 0 & -2 \\
		-1 & 2 & \dots & 0 & 0 & 0 & \dots & 0 & 1 \\
		-1 & 0 & \dots & 0 & 0 & 0 & \dots & -2 & 1 \\
		0 & -1 & \dots & 0 & 0 & 0 & \dots & 1 & 0 \\
		\vdots & \vdots & \ddots & \vdots & \vdots & \vdots & \ddots & \vdots & \vdots \\
		0 & 0 & \dots & 0 & 0 & 0 & \dots & 0 & 0 \\
		0 & 0 & \dots & 2 & 0 & 0 & \dots & 0 & 0 \\
		0 & 0 & \dots & 0 & 0 & -2 & \dots & 0 & 0 \\
		0 & 0 & \dots & -1 & 2 & 1 & \dots & 0 & 0 \\
		0 & 0 & \dots & -1 & -2 & 1 & \dots & 0 & 0 \\
	\end{pmatrix}
	\end{alignat*}
	\endgroup
	\caption{$4$ times the $\b{c}$-matrices of the wiggly pseudotriangulations of \cref{exm:specialWigglyPseudotriangulations}. The $i$th column of~$\b{c}(T)$ is $4$ times the $\b{c}$-vector~$\b{c}(\alpha, T)$ of the $i$th wiggly arc~$\alpha$ of~$T$ (ordered from bottom to top).}
	\label{table:specialCMatrices}
\end{table}
\end{example}

\begin{remark}
\label{rem:gcvectorsSymmetries}
Following \cref{rem:wigglyComplexAutomorphism}, note that the vertical and horizontal symmetries on wiggly arcs translate on $\b{g}$- and $\b{c}$-vectors to the isometries sending~$(x_1, x_2, \dots, x_{2n-1}, x_{2n}) \in \HH_{2n}$ to~$(x_{2n}, x_{2n-1}, \dots, x_2, x_1)$ and to~$(-x_2, -x_1, \dots, -x_{2n}, -x_{2n-1})$ respectively.
\end{remark}


\subsection{Wiggly fan}
\label{subsec:wigglyFan}

We now show that the $\b{g}$-vectors of \cref{def:gvectors} support a polyhedral fan realization of the wiggly complex~$\wigglyComplex_n$ illustrated in \cref{fig:wigglyFan} when~$n = 2$.
\afterpage{
\begin{figure}
\centerline{\includegraphics[scale=1.1]{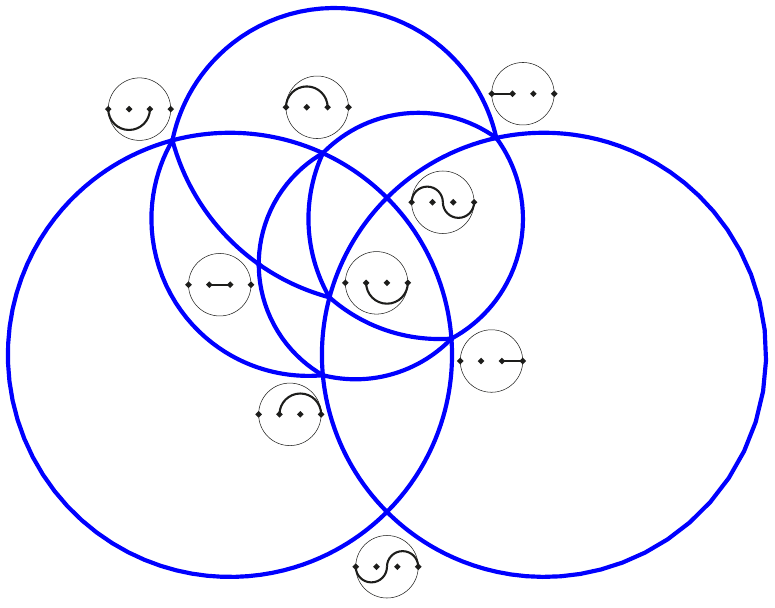}}
\caption{The wiggly fan~$\wigglyFan_2$, intersected with the unit sphere and projected stereographically to the plane. The vertices of the projection are rays of~$\wigglyFan_2$ and are labeled by the corresponding wiggly arcs. Compare with \cref{fig:wigglyComplex}\,(left).}
\label{fig:wigglyFan}
\end{figure}
}

\begin{theorem}
\label{thm:wigglyFan}
The collection of cones~$\R_{\ge 0} \b{g}(D) $ for all wiggly dissections~$D$ in the wiggly complex~$\wigglyComplex_n$ forms a complete simplicial fan of~$\HH_{2n}$, called the \defn{wiggly fan}~$\wigglyFan_n$.
\end{theorem}

To prove \cref{thm:wigglyFan}, we just need to check the conditions of \cref{prop:characterizationFan}.
We first find a vector contained in the cone~$\R_{\ge 0} \b{g}(T)$ for a single wiggly pseudotriangulation~$T$.

\begin{lemma}
\label{lem:-+...-+}
Let~$\b{v} \eqdef \sum_{k \in [n]} (\b{e}_{2k} - \b{e}_{2k-1}) = (-1, 1, \dots, -1, 1)$, let~$\alpha_0 \eqdef (0, 1, \varnothing, \varnothing)$, and for~${i \in [n]}$, let~$\underline{\alpha}_i \eqdef (0, i+1, [i], \varnothing)$ and~$\overline{\alpha}_i \eqdef (0, i+1, \varnothing, [i])$.
Then~$T_\leftarrow \eqdef \{\alpha_0\} \cup \bigcup_{i \in [n]} \{\underline{\alpha}_i, \overline{\alpha}_i\}$ is the only wiggly pseudotriangulation~$T$ such that~$\b{v} \in \R_{\ge 0} \b{g}(T)$.
\end{lemma}

\begin{proof}
Observe first that, for any wiggly arc~$\alpha \eqdef (i, j, A, B)$ and any~$k \in [n]$, we have
\[
\hat{\b{g}}(\alpha)_{2k-1} - \hat{\b{g}}(\alpha)_{2k} = \begin{cases} 2 & \text{if~$k = i$,} \\ -2 & \text{if~$k = j$,} \\ 0 & \text{otherwise}, \end{cases}
\qquad\text{and}\qquad
\hat{\b{g}}(\alpha)_{2k-1} + \hat{\b{g}}(\alpha)_{2k} = \begin{cases} 2 & \text{if~$k \in A$,} \\ -2 & \text{if~$k \in B$,} \\ 0 & \text{otherwise}. \end{cases}
\]

Assume now that~$\b{v} = \sum_\alpha \lambda_\alpha \, \hat{\b{g}}(\alpha)$ where the sum ranges over all wiggly arcs~$\alpha$ for some positive scalars~$\lambda_\alpha \ge 0$ such that~$\lambda_\alpha \cdot \lambda_{\alpha'} = 0$ for all incompatible wiggly arcs~$\alpha$ and~$\alpha'$.
For~$k \in [n]$, let~$\underline{\Lambda}_k$ (resp.~$\overline{\Lambda}_k$) denote the sum of~$\lambda_\alpha$ over the wiggly arcs~$\alpha \eqdef (i, j, A, B)$ with~$k \in A$ (resp.~$k \in B$).
Note that~$\underline{\Lambda}_k = \overline{\Lambda}_k$ since~$0 = \b{v}_{2k-1} + \b{v}_{2k} = \sum_\alpha \lambda_\alpha \big( \hat{\b{g}}(\alpha)_{2k-1} + \hat{\b{g}}(\alpha)_{2k} \big) = 2\underline{\Lambda}_k - 2\overline{\Lambda}_k$.

As~$\b{v}_1 - \b{v}_2 = -2$ and~$\hat{\b{g}}(\alpha)_1 - \hat{\b{g}}(\alpha)_2 \ge 0$ except if~$\alpha = \alpha_0$, we have~$\lambda_{\alpha_0} > 0$, and thus~$\lambda_{\alpha} = 0$ for all wiggly arcs~$\alpha$ starting at~$1$.
We then prove by induction on~$i \in [n-1]$ that~$\lambda_{\underline{\alpha}_i} \cdot \lambda_{\overline{\alpha}_i} > 0$ and that~$\lambda_\alpha = 0$ for all wiggly arcs starting at~$i+1$.
Consider~$i \in [n-1]$ and assume that we proved it for~$i-1$.
As~$\smash{\lambda_{\underline{\alpha}_{i-1}} \cdot \lambda_{\overline{\alpha}_{i-1}} > 0}$, we already have~$\lambda_\alpha = 0$ for all wiggly arcs~$\alpha$ ending at~$i+1$, except if~$\alpha \in \{\underline{\alpha}_i, \overline{\alpha}_i\}$.
As~$\b{v}_{2i+1} - \b{v}_{2i+2} = -2$, this implies that~$\smash{\lambda_{\underline{\alpha}_i} + \lambda_{\overline{\alpha}_i} = 1}$.
This already implies that~$\lambda_{\alpha} = 0$ for all wiggly arcs~$\alpha$ starting at~$i+1$.
Moreover, assume for instance that~$\lambda_{\overline{\alpha}_i} > 0$.
Hence, $\lambda_\alpha = 0$ for all wiggly arcs~$\alpha \eqdef (i, j, A, B)$ with~$i \in A$ and~$i+1 \in B$.
We thus obtain that~$\underline{\Lambda}_i - \lambda_{\underline{\alpha}_i} \le \underline{\Lambda}_{i+1}$ and~$\overline{\Lambda}_i - \lambda_{\overline{\alpha}_i} \ge \overline{\Lambda}_{i+1}$.
As~$\underline{\Lambda}_i = \overline{\Lambda}_i$ and~$\underline{\Lambda}_{i+1} = \overline{\Lambda}_{i+1}$, we thus obtain that~$\lambda_{\underline{\alpha}_i} \ge \lambda_{\overline{\alpha}_i} > 0$, hence that~$\lambda_{\underline{\alpha}_i} \cdot \lambda_{\overline{\alpha}_i} > 0$.
We conclude that~$\lambda_\alpha = 0$ for all wiggly arcs~$\alpha$ except if~$\alpha = \alpha_0$ or~$\alpha \in \{\underline{\alpha}_i, \overline{\alpha}_i\}$ for some~$i \in [n]$.

Consider now a wiggly pseudotriangulation~$T$ such that~$\b{v} \in \R_{\ge 0} \b{g}(T)$.
Write~$\b{v} = \sum_{\alpha \in T^\circ} \lambda_\alpha \, \b{g}(\alpha)$ with~$\lambda_\alpha \ge 0$ for all~$\alpha \in T^\circ$.
Let~$\mu \eqdef \sum_{\alpha \in T^\circ} \lambda_\alpha \dotprod{\hat{\b{g}}(\alpha)}{\one_{2n}}$.
As~$\b{v} = \sum_{\alpha \in T^\circ} \lambda_\alpha \, \hat{\b{g}}(\alpha) - \mu \, \one_{2n}$ and~$\one_{2n} = \hat{\b{g}} \big( (0, n+1, [n], \varnothing) \big) = - \hat{\b{g}} \big( (0, n+1, \varnothing, [n]) \big)$, we thus conclude from the previous paragraph that~$T = T_\leftarrow$. 
\end{proof}

We now describe the unique linear dependence between the $\b{g}$-vectors of two adjacent facets of the wiggly complex~$\wigglyComplex_n$.
The following statement is illustrated in \cref{fig:incompatible3}.

\begin{lemma}
\label{lem:linearDependences}
Let~$T$ and~$T'$ be two adjacent wiggly pseudotriangulations with~$T \ssm \{\alpha\} = T' \ssm \{\alpha'\}$, and let~$\beta$ and~$\beta'$ be the wiggly arcs defined in \cref{prop:uerp}.
Then~$\{\beta, \beta'\} \in T \cap T'$ and
\begin{itemize}
\item if $\alpha$ and~$\alpha'$ are not pointed, then~$\b{g}(\alpha) + \b{g}(\alpha') = \big( \b{g}(\beta) + \b{g}(\beta') \big) / 2$,
\item if~$\alpha$ and~$\alpha'$ are crossing, then~${\b{g}(\alpha) + \b{g}(\alpha') = \b{g}(\beta) + \b{g}(\beta')}$.
\end{itemize}
\end{lemma}

\begin{proof}
In both situations, we will just prove that the same linear dependence holds on the vectors~$\hat{\b{g}}(\alpha), \hat{\b{g}}(\alpha'), \hat{\b{g}}(\beta), \hat{\b{g}}(\beta')$, which will imply the statement as~$\pi$ is a linear map.

If~$\alpha$ and~$\alpha'$ are non pointed, then we have
\[
\alpha^+ \sqcup \alpha'^+ = \beta^+ \sqcup \{2j-1, 2j\} = \beta'^+
\qquad\text{and}\qquad 
\alpha^- \sqcup \alpha'^- = \beta^- = \beta'^- \sqcup \{2j-1, 2j\}.
\]
We thus obtain that
\[
\one_{\alpha^+} + \one_{\alpha'^+} = \one_{\beta^+} + \b{e}_{2j-1} + \b{e}_{2j} = \one_{\beta'^+}
\qquad\text{and}\qquad 
\one_{\alpha^-} + \one_{\alpha'^-} = \one_{\beta^-} = \one_{\beta'^-} + \b{e}_{2j-1} + \b{e}_{2j}.
\]
hence
\begin{align*}
2 \big( \hat{\b{g}}(\alpha) + \hat{\b{g}}(\alpha') \big)
& = 2(\one_{\alpha^+} - \one_{\alpha^-}) + 2(\one_{\alpha'^+} - \one_{\alpha'^-})
= 2(\one_{\alpha^+} + \one_{\alpha'^+}) - 2(\one_{\alpha^-} + \one_{\alpha'^-}) \\
& = (\one_{\beta^+} + \one_{\beta'^+} + \b{e}_{2j-1} + \b{e}_{2j}) - (\one_{\beta^-} + \one_{\beta'^-} + \b{e}_{2j-1} + \b{e}_{2j}) \\
& = (\one_{\beta^+} + \one_{\beta'^+}) - (\one_{\beta^-} + \one_{\beta'^-})
= (\one_{\beta^+} - \one_{\beta^-}) + (\one_{\beta^-} + \one_{\beta'^-})
= \hat{\b{g}}(\beta) + \hat{\b{g}}(\beta').
\end{align*}

If~$\alpha$ and~$\alpha'$ cross, then we have
\begin{gather*}
\alpha^+ \cup \alpha'^+ = \beta^+ \cup \beta'^+ \qquad\text{and}\qquad \alpha^+ \cap \alpha'^+ = \beta^+ \cap \beta'^+, \\
\alpha^- \cup \alpha'^- = \beta^- \cup \beta'^- \qquad\text{and}\qquad \alpha^- \cap \alpha'^- = \beta^- \cap \beta'^-.
\end{gather*}
We thus obtain that
\[
\one_{\alpha^+} + \one_{\alpha'^+} = \one_{\alpha^+ \cup \alpha'^+} + \one_{\alpha^+ \cap \alpha'^+} = \one_{\beta^+ \cup \beta'^+} + \one_{\beta^+ \cap \beta'^+} = \one_{\beta^+} + \one_{\beta'^+}
\]
and similarly~$\one_{\alpha^-} + \one_{\alpha'^-} = \one_{\beta^-} + \one_{\beta'^-}$, so that~$\hat{\b{g}}(\alpha) + \hat{\b{g}}(\alpha') =\hat{\b{g}}(\beta) + \hat{\b{g}}(\beta')$.
\end{proof}

\begin{corollary}
\label{coro:basis}
For any wiggly pseudotriangulation~$T$, the set~$\b{g}(T) \eqdef \set{\b{g}(\alpha)}{\alpha \in T^\circ}$ is a basis of~$\HH_{2n}$.
\end{corollary}

\begin{proof}
Observe that the $\b{g}$-matrix~$\b{g}(T_\leftarrow)$ presented in \cref{exm:specialGMatrices} has full rank (it is very easy to check for~$\b{g}(T_\leftarrow)$, but any other would do).
The statement thus follows from \cref{lem:linearDependences} and the connectedness of the flip graph on wiggly pseudotriangulations.
\end{proof}

Finally, we wrap up \cref{thm:wigglyFan}.

\begin{proof}[Proof of \cref{thm:wigglyFan}]
Directly follows from \cref{prop:characterizationFan,lem:-+...-+,lem:linearDependences}.
\end{proof}

\begin{remark}
\label{rem:wigglyFanSymmetries}
Following \cref{rem:wigglyComplexAutomorphism,rem:gcvectorsSymmetries}, note that the vertical and horizontal symmetries on wiggly arcs define linear automorphisms of the wiggly fan~$\wigglyFan_n$.
\end{remark}

\begin{remark}
By definition, each wall of the wiggly fan is normal to a $\b{c}$-vector.
Hence, the wiggly fan coarsens the \defn{$\b{c}$-vector fan}, defined by the arrangement of hyperplanes~$\set{\b{x} \in \HH_{2n}}{\dotprod{\b{x}}{\b{c}(\alpha, T}}$ for all wiggly arcs~$\alpha$ and wiggly pseudotriangulations~$T$ with~$\alpha \in T$.
\end{remark}


\subsection{Wigglyhedron}
\label{subsec:wigglyhedron}

We now construct the wigglyhedron illustrated in \cref{fig:wigglyhedron} when~${n = 2}$.
\afterpage{
\begin{figure}
\centerline{\includegraphics[scale=1.1]{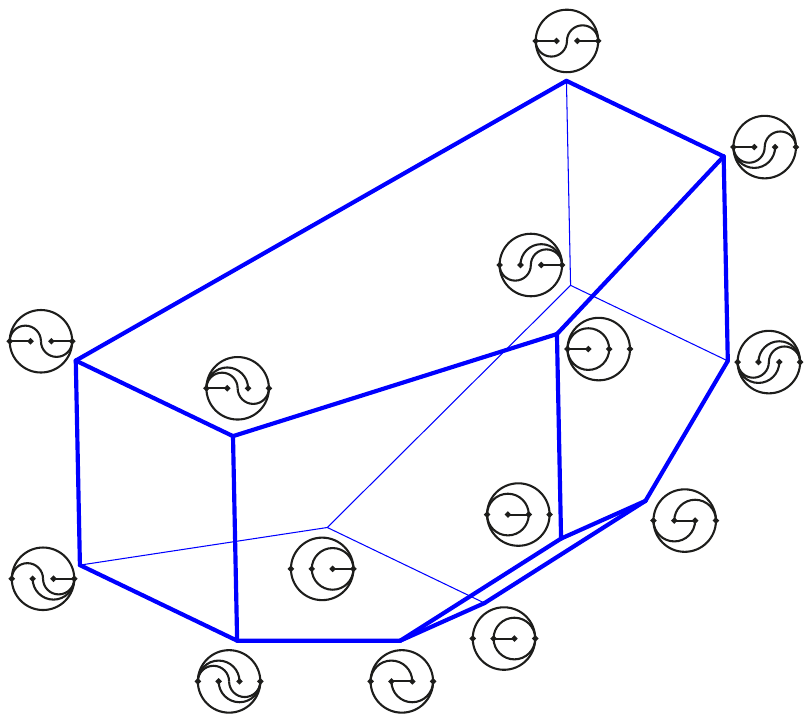}}
\caption{The wigglyhedron~$\wigglyhedron_2$. The vertices are labeled by the corresponding wiggly pseudotriangulations. Compare to \cref{fig:wigglyComplex}\,(right) and \cref{fig:wigglyLattice}\,(left).}
\label{fig:wigglyhedron}
\end{figure}
}

\begin{definition}
\label{def:incompatibilityDegree}
The \defn{incompatibility degree}~$\delta(\alpha, \alpha')$ of two wiggly arcs~$\alpha \eqdef (i, j, A, B)$ and~$\alpha' \eqdef (i', j', A', B')$ is given by
\begin{itemize}
\item $0$ if~$\alpha$ and~$\alpha'$ are pointed and non-crossing,
\item $1$ is~$\alpha$ and~$\alpha'$ are not pointed (\ie $i = j'$ or~$i' = j$),
\item the number of crossings of~$\alpha$ and~$\alpha'$ if they are crossing (\ie the number of~$p \in \{0, \dots, n\}$ with $p \in (A \cap B') \cup (\{i,j\} \cap B') \cup (A \cap \{i',j'\})$ and~$p+1 \in (A' \cap B) \cup (\{i',j'\} \cap B) \cup (A' \cap \{i,j\})$, or the opposite).
\end{itemize}
The \defn{incompatibility number} of a wiggly arc~$\alpha$ is~$\kappa(\alpha) \eqdef \sum_{\alpha'} \delta(\alpha, \alpha')$.
\end{definition}

\begin{remark}\label{rem:incompatibility-deg}
Note that~$\delta$ is indeed an incompatibility degree in the sense that
\begin{itemize}
\item $\alpha$ and~$\alpha'$ are compatible if and only if~$\delta(\alpha, \alpha') = 0$,
\item $\alpha$ and~$\alpha'$ are exchangeable if and only if~$\delta(\alpha, \alpha') = 1$.
\end{itemize}
Observe also that~$\delta$ (and thus~$\kappa$) is preserved by the automorphisms of \cref{rem:wigglyComplexAutomorphism}.
\end{remark}


\begin{example}
Observe that
$\kappa \big( (0, j, {[1,j[}, \varnothing) \big) 
= \kappa \big( (0, j, \varnothing, {[1,j[}) \big) 
= (2^j-1)(2^{n+1-j}-1)$.
\end{example}

\begin{lemma}
\label{lem:wallCrossingInequalities}
Let~$T$ and~$T'$ be two adjacent wiggly pseudotriangulations with~$T \ssm \{\alpha\} = T' \ssm \{\alpha'\}$, and let~$\beta$ and~$\beta'$ be the wiggly arcs defined in \cref{prop:uerp}.
Then
\begin{itemize}
\item if $\alpha$ and~$\alpha'$ are not pointed, then~$\kappa(\alpha) + \kappa(\alpha') > \big( \kappa(\beta) + \kappa(\beta') \big) / 2$,
\item if~$\alpha$ and~$\alpha'$ are crossing, then~${\kappa(\alpha) + \kappa(\alpha') > \kappa(\beta) + \kappa(\beta')}$.
\end{itemize}
\end{lemma}

\begin{proof}
Any crossing between a wiggly arc~$\gamma$ and~$\beta$ or~$\beta'$ translates to a crossing with~$\alpha$ or~$\alpha'$.
When~$\alpha$ and~$\alpha'$ are not pointed, a crossing with~$\alpha$ or~$\alpha'$ can correspond to a crossing with~$\beta$ and a crossing with~$\beta'$, so that~$\delta(\alpha, \gamma) + \delta(\alpha', \gamma) \ge \big( \delta(\beta, \gamma) + \delta(\beta', \gamma) \big) / 2$.
When~$\alpha$ and~$\alpha'$ are crossing, then there is a correspondence between the crossings with~$\alpha$ or~$\alpha'$ and the crossings with~$\beta$ or~$\beta'$, so that~$\delta(\alpha, \gamma) + \delta(\alpha', \gamma) \ge \delta(\beta, \gamma) + \delta(\beta', \gamma)$.
Moreover, both inequalities are strict for~$\gamma = \alpha'$ since~$\delta(\alpha,\alpha') = 1$ while~$\delta(\beta, \alpha') = \delta(\beta, \alpha') = 0$.
The result thus follows by summation.
\end{proof}

\begin{theorem}
\label{thm:wigglyhedron}
The wiggly fan~$\wigglyFan_n$ is the normal fan of a simplicial $(2n-1)$-dimensional polytope, called the \defn{wigglyhedron}~$\wigglyhedron_n$, and defined equivalently~as
\begin{itemize}
\item the intersection of the halfspaces~$\set{\b{x} \!\in\! \HH_{2n}\!}{\!\dotprod{\b{g}(\alpha)}{\b{x}} \! \le \! \kappa(\alpha)}$ for all internal wiggly~arcs~$\alpha$,
\item the convex hull of the points~$\b{p}(T) \eqdef \sum\limits_{\alpha \in T} \kappa(\alpha) \, \b{c}(\alpha, T)$ for all wiggly pseudotriangulations~$T$.
\end{itemize}
\end{theorem}

\begin{proof}
By \cref{prop:characterizationPolytopalFan,lem:linearDependences,lem:wallCrossingInequalities}, the wiggly fan is indeed the normal fan of the polytope~$\wigglyhedron_n$ defined by the inequalities~$\dotprod{\b{g}(\alpha)}{\b{x}} \le \kappa(\alpha)$ for all internal wiggly arcs~$\alpha$.
For any wiggly pseudotriangulation~$T$ and any wiggly arc~$\alpha$ of~$T$, we have~$\dotprod{\b{g}(\alpha)}{\b{p}(T)} = \sum_{\alpha' \in T} \kappa(\beta) \dotprod{\b{g}(\alpha)}{\b{c}(\alpha', T)} = \kappa(\alpha)$ since~$\b{g}(T)$ and~$\b{c}(T)$ are dual bases by \cref{def:cvectors}.
Hence, $\b{p}(T)$ is indeed the point of~$\wigglyhedron_n$ corresponding to~$T$, which is located at the intersection of the hyperplanes~$\dotprod{\b{g}(\alpha)}{\b{x}} = \kappa(\alpha)$ for all internal wiggly arcs~$\alpha$ in~$T$.
\end{proof}

In fact, the wigglyhedron even recovers the wiggly lattice of \cref{subsec:wigglyPermutations}.

\begin{proposition}
The Hasse diagram of the wiggly lattice~$\wigglyLattice_n$ is isomorphic to the graph of the wigglyhedron~$\wigglyhedron_n$ oriented in the direction
\(
\b{\omega} \eqdef \sum_{i \in [n]} 4ni (\b{e}_{2i-1} + \b{e}_{2i}) + \sum_{j \in [2n]} j \b{e}_j.
\)
\end{proposition}

\begin{proof}
Consider an increasing flip~$T \to T'$ between two wiggly pseudotriangulations~$T$ and~$T'$ with~$T \ssm \{\alpha\} = T' \ssm \{\alpha'\}$.
Denote by~$u$ and~$v$ the labels incident to~$\alpha$ and respectively below and above~$\alpha$ in the corner labeling of~$T$ described in \cref{def:bijection1}, and let~$\bar u \eqdef \lceil u/2 \rceil$ and~$\bar v \eqdef \lceil v/2 \rceil$.
We have~$u < v$ because the flip~$T \to T'$ is increasing, so that~$\bar u \le \bar v$.

Observe now that~${\b{p}(T') - \b{p}(T)}$ is orthogonal to the $(2n-2)$-dimensional cone of the wiggly fan generated by the $\b{g}$-vectors of the wiggly arcs in~$T \cap T'$, and points from the side containing~$\b{g}(\alpha)$ to the side containing~$\b{g}(\alpha')$.
Hence, $\b{p}(T') - \b{p}(T)$ is a positive multiple of~$-\b{c}(\alpha, T)$.
From \cref{def:cvectors}, we observe that
\begin{itemize}
\item $\dotprod{\b{c}(\alpha, T)}{\b{e}_{2i-1} + \b{e}_{2i}} = 1/2$ if~$i = \bar u \ne \bar v$, $-1/2$ if~$i = \bar v \ne \bar u$, and $0$ otherwise,
\item $|\dotprod{\b{c}(\alpha, T)}{\b{e}_j}| \le 1/2$ for all~$j \in [2n]$.
\end{itemize}
We now distinguish two situations.
If~$\bar u \ne \bar v$, then
\[
\dotprod{-\b{c}(\alpha, T)}{\b{\omega}} = \sum_{i \in [n]} 4ni \dotprod{-\b{c}(\alpha, T)}{\b{e}_{2i-1} + \b{e}_{2i}} + \sum_{j \in [2n]} j \dotprod{-\b{c}(\alpha, T)}{\b{e}_j} \ge 2n(\bar v - \bar u) - n > 0.
\]
If~$\bar u = \bar v$, then~$u = 2k-1$ and~$v = 2k$ for some~$k \in [n]$, and we have
\[
\dotprod{\b{c}(\alpha, T)}{\b{\omega}} = \!\! \sum_{i \in [n]} 4ni \dotprod{-\b{c}(\alpha, T)}{\b{e}_{2i-1} + \b{e}_{2i}} + \!\! \sum_{j \in [2n]} j \dotprod{-\b{c}(\alpha, T)}{\b{e}_j} = \frac{2k}{2}-\frac{2k-1}{2} = 1/2 > 0.
\]
We conclude in both cases that~$\dotprod{\b{c}(\alpha, T)}{\b{\omega}} > 0$, so that we recover the increasing flip graph by orienting the graph of the wigglyhedron~$\wigglyhedron_n$ in direction~$\b{\omega}$.
\end{proof}

\begin{remark}
\label{rem:wigglyhedronSymmetries}
Following \cref{rem:wigglyComplexAutomorphism,rem:gcvectorsSymmetries,rem:wigglyFanSymmetries}, note that the vertical and horizontal symmetries on wiggly arcs define isometries of the wigglyhedron~$\wigglyhedron_n$.
\end{remark}


\section{Further topics}
\label{sec:furtherTopics}

In this section, we explore two further topics on the wigglyhedron.
We first observe that any (type~$A$) Cambrian associahedron~\cite{Reading-CambrianLattices,HohlwegLange} is normally equivalent to a face of the wigglyhedron~$\wigglyhedron_n$ (\cref{subsec:CambrianConsiderations}).
We then raise the question to extend our construction to wiggly complexes of arbitrary planar point sets (\cref{subsec:planarPointSets}).


\subsection{Cambrian considerations}
\label{subsec:CambrianConsiderations}

In this section, we fix a signature~$\delta \eqdef \delta_1 \dots \delta_n \in \{+,-\}^n$. By convention, we set~$\delta_0 \eqdef \delta_{n+1} \eqdef 0$.
We connect the $\delta$-Cambrian lattice~\cite{Reading-CambrianLattices} to the wiggly lattice of \cref{sec:combinatorics} and the $\delta$-associahedron of~\cite{HohlwegLange} to the wigglyhedron of \cref{sec:geometry}.
\cref{def:CambrianTriangulations,def:CambrianPermutations,def:CambrianWigglyPseudotriangulations,def:CambrianWigglyPermutations,prop:CambrianLattice} are illustrated in \cref{fig:wigglyCambrian}.

\begin{figure}
\centerline{\includegraphics[scale=1.1]{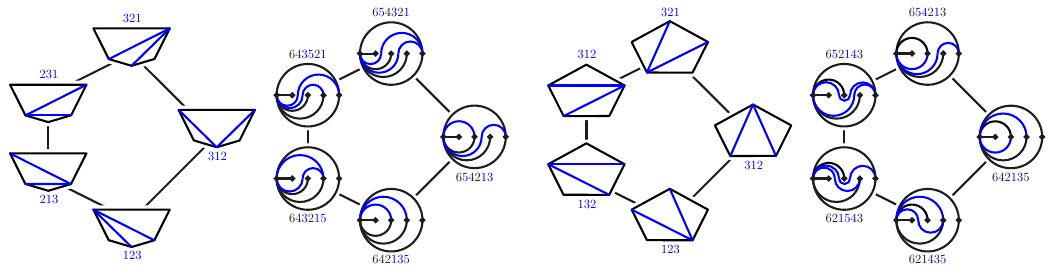}}
\vspace{-.3cm}
\centerline{\includegraphics[scale=1.1]{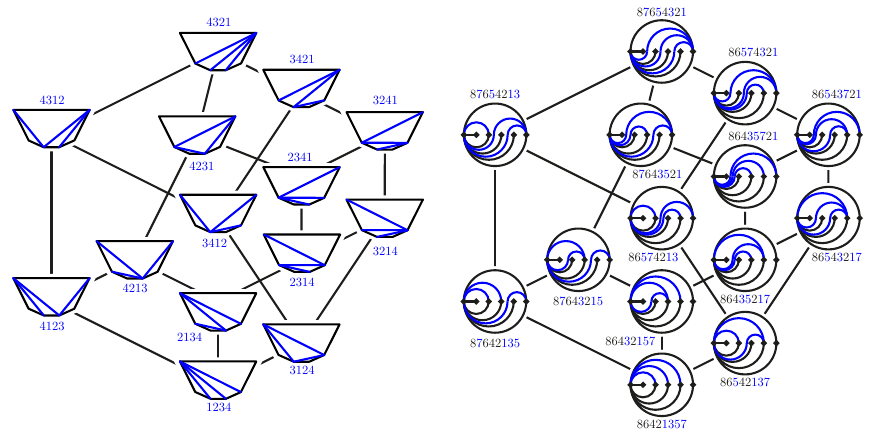}}
\vspace{-.5cm}
\centerline{\includegraphics[scale=1.1]{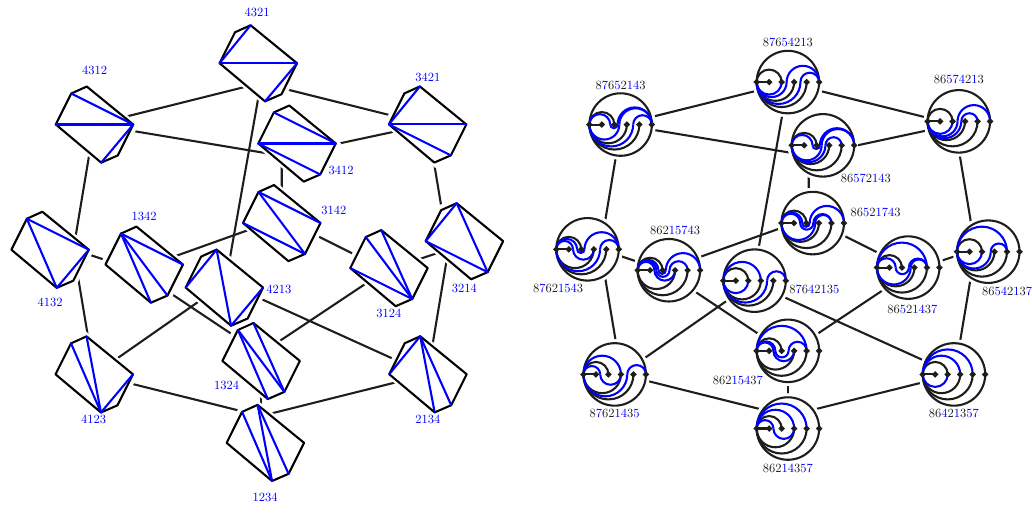}}
\caption{Cambrian lattices as intervals of wiggly lattices.}
\label{fig:wigglyCambrian}
\end{figure}

\begin{definition}
\label{def:CambrianTriangulations}
The \defn{$\delta$-gon} is a convex $(n+2)$-gon with a point with abscissa~$i$ and with ordinate of the same sign as~$\delta_i$  for all~$0 \le i \le n+1$.
A \defn{$\delta$-triangulation} is a triangulation of the $\delta$-gon.
The \defn{$\delta$-triangulation lattice} is the transitive closure of the slope increasing flip graph on $\delta$-triangulations.
\end{definition}

\begin{definition}
\label{def:CambrianPermutations}
A \defn{$\delta$-permutation} is a permutation of~$[n]$ which avoids~$\cdots ik \cdots j \cdots$ with~$i < j < k$ and~$\delta_j = -$, and~$\cdots j \cdots ki \cdots$ with~$i < j < k$ and~$\delta_j = +$.
The \defn{$\delta$-permutation lattice} is the sublattice of the weak order induced by $\delta$-permutations.
\end{definition}

\begin{definition}
\label{def:CambrianWigglyPseudotriangulations}
The \defn{$\delta$-wiggly arcs} are the arcs~$(0, j, {[1,j[}, \varnothing)$ for~$\delta_j = {+}$ and~$(0, j, \varnothing, {[1,j[})$ for~$\delta_j = {-}$.
A \defn{$\delta$-wiggly pseudotriangulation} is a wiggly pseudotriangulation containing all $\delta$-wiggly arcs.
The \defn{$\delta$-wiggly pseudotriangulation lattice} is the transitive closure of the increasing flip graph on $\delta$-wiggly pseudotriangulations.
\end{definition}

\begin{definition}
\label{def:CambrianWigglyPermutations}
A \defn{$\delta$-wiggly permutation} is a wiggly permutation~$\sigma$ of~$[2n]$ such that
\[
\delta_j = {+} \quad \Longrightarrow \quad \sigma^{-1}(i) \le \sigma^{-1}(2j-1)
\qquad\text{and}\qquad
\delta_j = {-} \quad \Longrightarrow \quad \sigma^{-1}(2j) \le \sigma^{-1}(i)
\]
for all~$1 \le i \le 2j \le 2n$.
The \defn{$\delta$-wiggly permutation lattice} is the interval of the wiggly lattice~$\wigglyLattice_n$ given by $\delta$-wiggly permutations.
\end{definition}

\begin{remark}
We had to make a choice in \cref{def:CambrianWigglyPseudotriangulations,def:CambrianWigglyPermutations}.
For~$\delta_j = {+}$, we could replace ``$(0, j, {[1,j[}, \varnothing)$'' by~``$(j, n+1, {]j,n]}, \varnothing)$'' in \cref{def:CambrianWigglyPseudotriangulations}, and ``$\sigma^{-1}(i) \le \sigma^{-1}(2j-1)$ for~$1 \le i \le 2j$'' by ``$\sigma^{-1}(i) \le \sigma^{-1}(2j-1)$ for~$2j-1 \le i \le 2n$'' in \cref{def:CambrianWigglyPermutations}.
Independently, for~$\delta_j = {-}$, we could replace ``$(0, j, \varnothing, {[1,j[})$'' by~``$(j, n+1, \varnothing, {]j,n]})$'' in \cref{def:CambrianWigglyPseudotriangulations}, and ``$\sigma^{-1}(2j) \le \sigma^{-1}(i)$ for~$1 \le i \le 2j$'' by ``$\sigma^{-1}(2j-1) \le \sigma^{-1}(i)$ for~$2j-1 \le i \le 2n$'' in \cref{def:CambrianWigglyPermutations}.
\cref{prop:CambrianLattice,prop:CambrianFan} would still hold.
\end{remark}

\begin{remark}
\label{rem:CambrianSymmetries}
Following \cref{rem:wigglyComplexAutomorphism,rem:gcvectorsSymmetries,rem:wigglyFanSymmetries,rem:wigglyhedronSymmetries}, note that the vertical and horizontal symmetries send $\delta \eqdef \delta_1 \dots \delta_n$ to~$\delta_n \cdots \delta_1$ and $(-\delta_1) \cdots (-\delta_n)$ respectively.
\end{remark}

\begin{proposition}
\label{prop:CambrianLattice}
For any signature~$\delta \in \{+,-\}^n$, the $\delta$-triangulation lattice, $\delta$-permutation lattice, \mbox{$\delta$-wiggly} pseudotriangulation lattice, and $\delta$-wiggly permutation lattice are all isomorphic, and are known as the \defn{$\delta$-Cambrian lattice}~\cite{Reading-CambrianLattices}.
\end{proposition}

\begin{proof}
We only describe the following bijections
\begin{center}
\begin{tikzcd}[column sep=3cm]
\text{$\delta$-triangulations} \arrow[d] & \text{$\delta$-permutations} \arrow[l] \\
\text{$\delta$-wiggly pseudotriangulations} \arrow[r] & \text{$\delta$-wiggly permutations} \arrow[u]
\end{tikzcd}
\end{center}
illustrated in \cref{fig:wigglyCambrian}.
We let the readers convince themselves that they are well-defined, that they compose to the identity, and that they preserve the flip structure. We also invite the reader to work out the explicit description of the 8 missing arrows.
For~$I \subseteq [n]$, we set~$I^+ \eqdef \set{i \in I}{\delta_i = {+}}$ and~$I^- \eqdef \set{i \in I}{\delta_i = {-}}$.

\para{From $\delta$-triangulations to $\delta$-wiggly pseudotriangulations}
Replace each diagonal~$(i,j)$ of the $\delta$-gon by the wiggly arc~$(0, j, [i] \cup {]i,j[}^-, {]i,j[}^+)$ if~$\delta_i \ge 0$ and~$(0, j, {]i,j[}^-, [i] \cup {]i,j[}^+)$ if~$\delta_i \le 0$.

\para{From $\delta$-wiggly pseudotriangulations to $\delta$-wiggly permutations}
Specialization of \cref{def:bijection1}.

\para{From $\delta$-wiggly permutations to $\delta$-permutations}
For all~$j \in [n]$, erase~$2j-1$ and replace~$2j$ by~$j$ if~$\delta_j = {+}$, and erase~$2j$ and replace~$2j-1$ by~$j$ if~$\delta_j = {-}$.

\para{From $\delta$-permutations to $\delta$-triangulations}
Already described in~\cite{Reading-CambrianLattices}.
Given a permutation~$\pi$, construct the union of the upper hulls of the point sets~$[n]^- \symdif \pi([j])$ for all~$j \in \{0, \dots, n\}$.
\end{proof}

Recall that the $\delta$-Cambrian lattice can be realized by the $\delta$-Cambrian fan~\cite{ReadingSpeyer} and the $\delta$-associahedron of~\cite{HohlwegLange}.
Namely, to each diagonal~$(i,j)$ of the $\delta$-gon, we associate the subset~$X(\delta)$ of points of the $\delta$-gon distinct from~$0$ and~$n+1$, and located below the line~$(i,j)$, including~$i$ if and only if~$\delta_i = {+}$ and~$j$ if and only if~$\delta_j = {+}$.
The \defn{$\delta$-Cambrian fan} is the fan with a ray~$\b{g}(d) \eqdef \pi(\one_{X(d)})$ for each internal diagonal~$d$ of the $\delta$-gon and a cone~$\R_{\ge0} \setangle{\b{g}(d)}{d \in D^\circ}$ for each dissection~$D$ of the $\delta$-gon.
The \defn{$\delta$-associahedron}~$\Asso_\delta$ is the polytope defined by the equality~$\dotprod{\one_n}{\b{x}} = \binom{n+1}{2}$ and the inequalities~$\dotprod{\b{g}(d)}{\b{x}} \ge \binom{|X(d)|+1}{2}$ for all internal diagonals~$d$ of the $\delta$-gon.

\begin{proposition}
\label{prop:CambrianFan}
The $\delta$-associahedron~$\Asso_\delta$ is normally equivalent to the face of the wigglyhedron~$\wigglyhedron_n$ corresponding to the wiggly pseudodissection formed by the $\delta$-wiggly arcs.
\end{proposition}

\begin{proof}
Following the proof of \cref{prop:CambrianLattice}, consider the map~$\zeta$ sending a diagonal~$(i,j)$ of the $\delta$-gon to the wiggly arc $(0, j, [i] \cup {]i,j[}^-, {]i,j[}^+)$ if~$\delta_i \ge 0$ and~$(0, j, {]i,j[}^-, [i] \cup {]i,j[}^+)$ if~$\delta_i \le 0$.
Let~$a \eqdef (i,j)$ and~$a' \eqdef (i',j')$ be two crossing diagonals of the~$\delta$-gon.
Then any $\delta$-triangulations~$T$ and~$T'$ with~$T \ssm \{a\} = T' \ssm \{a'\}$ contain the diagonals~$b \eqdef (i',j)$ and~$b' \eqdef (i,j')$ of the $\delta$-gon, and~$\b{g}(a) + \b{g}(a') = \b{g}(b) + \b{g}(b')$.
Moreover, if~$\alpha \eqdef \zeta(a)$, $\alpha' \eqdef \zeta(a')$, $\beta \eqdef \zeta(b)$ and~$\beta' \eqdef \zeta(b')$, then~$\b{g}(\alpha) + \b{g}(\alpha') = \b{g}(\beta) + \b{g}(\beta')$.
We thus conclude that $\zeta$ preserves the linear dependences among $\b{g}$-vectors.
Hence, there is a linear map~$Z : \HH_n \to \HH_{2n}$ such that~$Z \big( \b{g}(d) \big) = \b{g} \big( \zeta(d) \big)$ for all diagonals~$d$ of the $\delta$-gon.
This map is a normal equivalence from the $\delta$-associahedron~$\Asso_\delta$ to the face of the wigglyhedron~$\wigglyhedron_n$ corresponding to the wiggly pseudodissection formed by the $\delta$-wiggly~arcs.
\end{proof}


\begin{remark}
In their work on polytopal realizations of $\b{g}$-vector fans of finite type cluster algebras~\cite{HohlwegPilaudStella}, C.~Hohlweg, V.~Pilaud and S.~Stella already constructed a \defn{universal associahedron}~$\polytope{U}_n$ which somehow contains all Cambrian fans (actually all $\b{g}$-vector fans from any initial cluster seed).
More precisely, the $\delta$-Cambrian fan for each~$\delta \in \{+,-\}^n$  appears as a section of the normal fan of the universal associahedron~$\polytope{U}_n$, which is a $(n^2+n-4)/2$-dimensional polytope.
In contrast, \cref{prop:CambrianFan} states that the $\delta$-Cambrian fan for each~$\delta \in \{+,-\}^n$ appears as a link of the normal fan of the wigglyhedron~$\wigglyhedron_n$, which is a $(2n-1)$-dimensional polytope.
There seem to be no connection between these two universal polytopes.
However, this certainly raises the question to construct analogues of the wigglyhedron for all finite Coxeter groups.
\end{remark}


\subsection{Wiggly pseudotriangulations of planar point sets}
\label{subsec:planarPointSets}

We now consider \cref{thm:wigglyhedron,conj:Hamiltonian} for wiggly pseudotriangulations of arbitrary point sets in the plane (neither necessarily aligned, nor necessarily in general position).

\begin{definition}
\label{def:wigglyComplexPointSet}
Fix a point set~$P$ of the plane, and an arbitrary total order~$<$ on~$P$.
A \defn{wiggly arc} is a quadruple~$(p,q,R,S)$ where~$p < q \in P$ and the sets~$R$ and~$S$ form a partition of the points of~$P$ located in the open segment joining~$p$ to~$q$.
Two wiggly arcs~$(p,q,R,S)$ and~$(p',q',R',S')$ are \defn{crossing} if 
\begin{itemize}
\item either the segments~$[p,q]$ and~$[p',q']$ cross, 
\item or~$(R \cap S') \cup (\{p,q\} \cap S') \cup (R \cap \{p',q'\}) \ne \varnothing \ne (R' \cap S) \cup (\{p',q'\} \cap S) \cup (R' \cap \{p,q\})$.
\end{itemize}
A set~$X$ of wiggly arcs is \defn{pointed} if for any~$p \in P$, the wiggly arcs of~$X$ with an endpoint at~$p$ generate a pointed cone.
The \defn{wiggly complex}~$\wigglyComplex_P$ is the simplicial complex of pairwise pointed and non-crossing subsets of wiggly arcs.
Note that the boundary wiggly arcs are irrelevant, which allows us to consider a reduced wiggly complex~$\wigglyComplex_P$ induced by internal wiggly arcs.
A \defn{wiggly pseudotriangulation} of~$P$ is a facet of~$\wigglyComplex_P$.
The \defn{wiggly flip graph}~$\wigglyFlipGraph_P$ is the adjacency graph of the facets of~$\wigglyComplex_P$.
See \cref{fig:wigglyComplexSquarre} for an illustration.
\begin{figure}[!h]
\centerline{\includegraphics[scale=1]{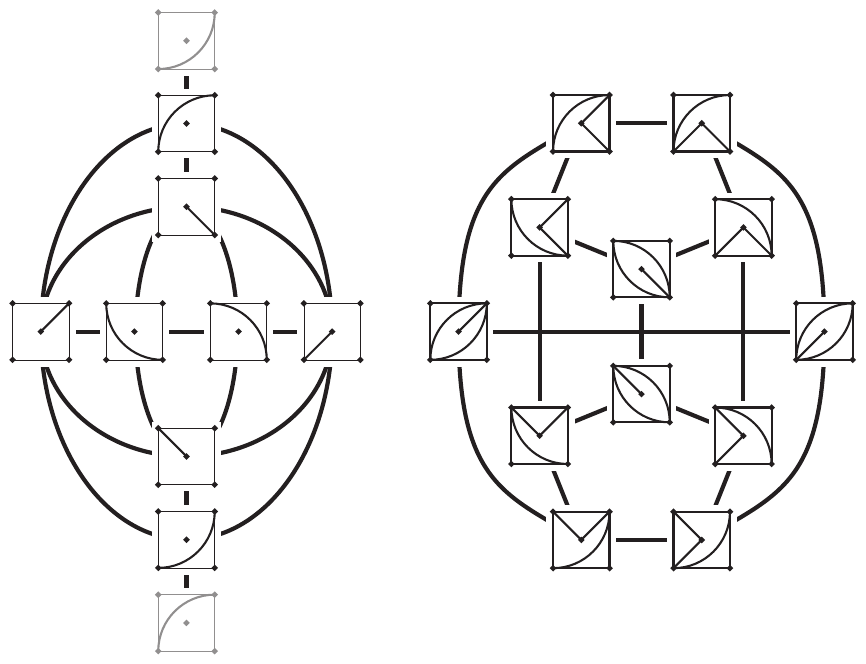}}
\caption{The wiggly complex~$\wigglyComplex_P$ (left) and the wiggly flip graph~$\wigglyFlipGraph_P$ (right) of a point set~$P$.}
\label{fig:wigglyComplexSquarre}
\end{figure}
\end{definition}

Note that \cref{def:wigglyComplexPointSet} recovers two specific situations, namely
\begin{itemize}
\item \cref{def:wigglyComplex,def:wigglyFlipGraph} in the case of aligned points, and
\item the pseudodissection complex of~$P$ and pseudotriangulation flip graph of~$P$ in the case of points~$P$ in general position. We refer to the original article of M.~Pocchiola and G.~Vegter~\cite{PocchiolaVegter} and to the nice survey of G.~Rote, F.~Santos and I.~Streinu~\cite{RoteSantosStreinu-pseudotriangulations}.
\end{itemize}

Observe that it is not even clear from \cref{def:wigglyComplexPointSet} that the wiggly complex~$\wigglyComplex_P$ is a pure pseudomanifold (hence that the wiggly flip graph~$\wigglyFlipGraph_P$ is well-defined).
However, we make the following ambitious conjecture.

\begin{conjecture}
\label{conj:polytopality}
For any point set~$P$ in the plane, the wiggly complex~$\wigglyComplex_P$ is the boundary complex of a simplicial polytope.
\end{conjecture}

In \cref{conj:polytopality}, the case of aligned points is given by the wigglyhedron of \cref{thm:wigglyhedron}, while the case of points in general position is given by the pseudotriangulation polytope of~\cite{RoteSantosStreinu-polytope}.
This raises in particular the following question.

\begin{question}
Can the construction of \cref{thm:wigglyhedron} be adapted to provide a more combinatorial construction of the polytope of pseudotriangulations~\cite{RoteSantosStreinu-polytope}, that would only depend on the order type (\aka oriented matroid~\cite{BjornerLasVergnasSturmfelsWhiteZiegler}) of the point configuration?
\end{question}

Using the duality between lines in the plane and points in the M\"obius strip, the pseudotriangulations of~\cite{PocchiolaVegter,RoteSantosStreinu-pseudotriangulations} were interpreted in~\cite{PilaudPocchiola} as pseudoline arrangements on sorting networks.
It is natural to look for the analogue of this interpretation in the wiggly setting.

\begin{question}
Is there a dual interpretation of wiggly pseudotriangulations as some sort of pseudoline arrangements? 
\end{question}

This dual interpretation enables \cite{PilaudPocchiola} to consider the pseudotriangulations of~\cite{PocchiolaVegter,RoteSantosStreinu-pseudotriangulations} and the multitriangulations of~\cite{PilaudSantos-multitriangulations} under the same roof, and to define multi-pseudo-triangulations.
The extension to the wiggly case is a natural question.

\begin{question}
Is there a multi wiggly complex?
\end{question}

Moreover, pseudoline arrangements on sorting networks can also be interpreted as (type~$A$) subword complexes of~\cite{KnutsonMiller-subwordComplex}, which in turn extend to arbitrary finite Coxeter group.
See \cite{Stump, CeballosLabbeStump} for the study of subword complexes corresponding to cluster complexes and multi-cluster complexes.

\begin{question}
Can the wiggly complex be extended to arbitrary finite Coxeter groups?
\end{question}


Finally, it is natural to consider the extension of \cref{conj:Hamiltonian}.

\begin{question}
\label{qu:Hamiltonian}
Is the wiggly flip graph~$\wigglyFlipGraph_P$ Hamiltonian for any point set~$P$ in the plane?
\end{question}


\section{Categorical interpretation of wiggly pseudotriangulations}
\label{sec:categoricalInterpretation}

The aim of this section is to explain a categorical interpretation of wiggly arcs.
We first present a particular triangulated category (\cref{subsec:category}), and recall results of Khovanov--Seidel~\cite{kho.sei:02} (\cref{subsec:KSCurveModel}) that interpret the objects of this category as wiggly arcs, such that intersections between wiggly arcs correspond to morphisms in this category.
We then explain how wiggly pseudotriangulations govern the combinatorial structure of a natural categorical decomposition of objects (\cref{subsec:cohomologyFiltrations}).
We also discuss categorical conditions for two wiggly arcs to be compatible as in~\cref{def:compatible} (\cref{subsec:categoricalCompatibility}).


\subsection{A triangulated category associated to a quiver}
\label{subsec:category}

We begin by describing the category of interest, which is the bounded homotopy category of finitely-generated graded projective modules over a finite-dimensional algebra \(A_m\). 
The algebra~\(A_m\), called the \defn{zigzag algebra}, is a quotient of the path algebra of the quiver \(Q_m\) with \(m\) vertices depicted in~\cref{fig:am-quiver}.
Let \(kQ_m\) be the path algebra of \(Q_m\), graded by path length.
Let~\(e_i\) denote the idempotent path at the vertex \(i\), and let \((i, j)\) denote the arrow from vertex \(i\) to a (neighbouring) vertex \(j\).
A longer path, which is a composition of arrows, will be denoted by~\((i_1, i_2, \dots, i_n)\).

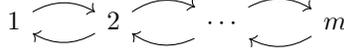
\begin{figure}[h]
	\centering
	\begin{tikzcd}
		1 \arrow[bend left]{r} & 2 \arrow[bend left]{r}\arrow[bend left]{l}& \cdots \arrow[bend left]{r}\arrow[bend left]{l}& m \arrow[bend left]{l}
	\end{tikzcd}
	\caption{The quiver \(Q_m\). The algebra \(A_m\) is a quotient of the path algebra \(kQ_m\).}
	\label{fig:am-quiver}
\end{figure}

\begin{definition}
The \defn{zigzag algebra} \(A_m\) corresponding to the quiver \(Q_m\) depicted in~\cref{fig:am-quiver} is the quotient of \(kQ_m\) by all paths of length at least \(3\), together with the relations \((i-1, i, i+1) = 0\), \((i+1, i, i-1)=0\) and \((i, i-1, i) = (i, i+1, i)\) for all~\(1 < i < m\).
\end{definition}

\begin{remark}
The definition of the zigzag algebra as used here originally appeared in~\cite[Sect.~3]{hue.kho:01}.
This is the ``un-signed'' zigzag algebra.
It is technically better to use a ``signed'' zigzag algebra, and the comparison between the two is discussed in~\cite[Rem.~6.6]{bap.deo.lic:20}.
As the two versions are isomorphic in our example, we ignore this technicality and use the easier definition.
\end{remark}

Let \(P_i\) denote the indecomposable projective \(A_m\)-module defined as \(P_i = A_m\cdot e_i\).
For any integer~\(j\), let \(P_i \langle j \rangle\) denote its graded shift by \(j\).

\begin{definition}
The category \(\mathcal{C}_m\) is defined to be the homotopy category of bounded complexes of finitely-generated graded projective \(A_m\)-modules.
\end{definition}
Let \(D^b(A_m\text{-grmod})\) denote the bounded derived category of graded \(A_m\)-modules.
Note that there is an embedding of \(A_m\text{-grmod}\) into \(D^b(A_m\text{-grmod})\), which sends any \(A_m\)-module as a complex of \(A_m\)-modules that is concentrated in degree \(0\).
In particular, the objects \(P_i\langle j \rangle\) can be regarded as objects of \(D^b(A_m\text{-grmod})\).
\(\mathcal{C}_m\) is thus the smallest full triangulated subcategory of \(D^b(A_m\text{-grmod})\) containing all the objects \(P_i\langle j \rangle\).

For further discussion on the construction, see \eg~\cite[Sect.~2.3.3]{bap.deo.lic:22}.
The category \(\mathcal{C}_m\) is \(k\)-linear and triangulated, where \([1]\) denotes the triangulated shift.
For \(X, Y \in \mathcal{C}_m\), let \(\Hom^{g,h}(X,Y)\) denote the space of morphisms in \(\mathcal{C}_m\) of internal degree \(g\) and homological degree \(h\).
In other words, we have
\[
\Hom^{g,h}(X,Y) = \Hom(X,Y\langle g \rangle[h]).
\]
Furthermore, for any integers \(a\) and \(b\), we have
\[
\Hom^{g,h}(X\langle g' \rangle[h'], Y\langle g' \rangle[h']) = \Hom^{g,h}(X,Y).
\]

\begin{remark}
\label{rem:generating-morphisms}
The generating objects \(P_i\) (and hence their graded and homological shifts) satisfy \(\Hom^{0,0}(P_i,P_i) = \Hom^{2,0}(P_i,P_i) = \Hom^{1,0}(P_i,P_{i\pm1}) = k\) and \(\Hom^{g,h}(P_i, P_j) = 0\) in all other cases.
\end{remark}

We now describe the \(t\)-structure induced by the full subcategory of \defn{linear complexes} of \(\mathcal{C}_m\).
A variant of this construction is also described in~\cite[Sect.~2.5]{lic.que:21}.
It is convenient to think of an object \(P^{\bullet}\) of \(\mathcal{C}_m\) as a bigraded (projective) \(A_m\)-module equipped with a differential \(d \colon P^{\bullet} \to P^{\bullet}\).
More precisely, we can write
\[
P^{\bullet} = \bigoplus_{g,h,i} M_i\langle g \rangle[h],
\]
where \(g\) is the internal grading shift, \(h\) is the homological grading shift, and 
each object \(M_i\) is a direct sum of copies of objects in \(\{P_1,\ldots, P_m\}\).
Set \(P^h\) to be the direct sum of all summands that have homological grading shift \(h\).
That is, in the notation above, we have
\[
P^{h} = \bigoplus_{g,i} M_i\langle g \rangle[h].
\]

\begin{warning}
  By convention, if \(M_i\) is an \(A_m\)-module, then the object \(M_i[h]\) is placed in homological degree \(-h\).
  Since the differential \(d\) raises the homological degree by one, it lowers the homological shift by one.
  That is, we have
  \[d \colon M_i\langle g \rangle[h] \to \bigoplus_{g',i'} M_{i'}\langle g' \rangle[h-1].\]
  In other words, \(d\) maps \(P^h\) to \(P^{h-1}\).
\end{warning}

\begin{definition}
\label{def:linear-complex}
Let \(P^{\bullet}\) be an object of \(\mathcal{C}_m\), described as a bigraded module
\[
P^{\bullet} = \bigoplus M_i\langle g \rangle[h]
\]
together with an internal differential \(d \colon P^{\bullet} \to P^{\bullet}\) as above.
We say that \(P^{\bullet}\) has \defn{level \(\ell\)} for some~\(\ell \in \mathbb{Z}\) if for each bigraded piece \(M_i \langle g \rangle[h]\) of \(P^{\bullet}\), we have \(g + h = \ell\).
We say that \(P^{\bullet}\) is \defn{linear} if \(P^{\bullet}\) has level \(\ell = 0\).
\end{definition}

For any integer \(N\), let \(\mathcal{C}_m^{\le N} \subset \mathcal{C}_m\) be the smallest full subcategory of \(\mathcal{C}_m\) closed under isomorphism, containing all complexes of level \(\geq -N\).
For any integer \(N\), let \(\mathcal{C}_m^{\geq N} \subset \mathcal{C}_m\) be the smallest full subcategory of \(\mathcal{C}_m\) closed under isomorphism, containing all complexes of level \(\leq -N\).

We refer the reader to \eg~\cite[Sect.~1.3]{bel.ber.del:82} for the definition of a \(t\)-structure on a triangulated category.
The following proposition is easily checked using the calculations in~\cref{rem:generating-morphisms} and standard techniques involving categories of complexes; we omit the proof.

\begin{proposition}
The pair of subcategories \((\mathcal{C}_m^{\leq 0}, \mathcal{C}_m^{\geq 0})\) forms a bounded \(t\)-structure on \(\mathcal{C}_m\).
\end{proposition}

We call this \(t\)-structure the \defn{standard \(t\)-structure} on \(\mathcal{C}_m\).
Set \(\mathcal{A}_m \eqdef \mathcal{C}^{\leq 0}_m \cap \mathcal{C}_m^{\geq 0}\) to be the heart of this \(t\)-structure; this is the smallest full subcategory of \(\mathcal{C}_m\) closed under isomorphism, which contains all linear complexes.
It follows that \(\mathcal{A}_m\) is abelian.

Also record the following statements for future use; the first may be found in~\cite[Sect.~1.3]{bel.ber.del:82}.

\begin{defprop}
\label{defprop:cohomology}
Let \((\mathcal{C}^{\le 0}, \mathcal{C}^{\geq 0})\) be a \(t\)-structure on a triangulated category \(\mathcal{C}\) with heart~\(\mathcal{A} \eqdef \mathcal{C}^{\leq 0} \cap \mathcal{C}^{\geq 0}\).
For each integer \(i\), we have the \(i\)th truncation functors
\[
\tau^{\leq i} \colon \mathcal{C} \to \mathcal{C}^{\leq i} \quad\text{and}\quad \tau^{\geq i} \colon \mathcal{C} \to \mathcal{C}^{\geq i},
\]
which are right and left adjoint respectively to the inclusion functors \(\mathcal{C}^{\leq 0} \to \mathcal{C}\) and \(\mathcal{C}^{\geq 0} \to \mathcal{C}\).
The \(i\)th \defn{cohomology functor} is defined to be the composition \(\tau^{\leq i} \circ \tau^{\geq i}\).
It is canonically isomorphic to the composition \(\tau^{\geq i} \circ \tau^{\leq i}\), and maps an object of \(\mathcal{C}\) to an object of \(\mathcal{A}\).
\end{defprop}

The next statement, easy to check from the definitions, is mentioned \eg in~\cite[Sect.~3.3]{bayer2011tour}.

\begin{defprop}
\label{defprop:cohomology-filtration}
Let \(\mathcal{C}\) be any triangulated category, and \(\mathcal{A} \subset \mathcal{C}\) the heart of a bounded \(t\)-structure on \(\mathcal{C}\).
Let \(X\) be a non-zero object of \(\mathcal{C}\), and let \(k\) and \(\ell\) be the maximum and minimum integers respectively for which \(H^k(X)\) and \(H^\ell(X)\) are non-zero.
Then \(X\) has a unique \defn{cohomology filtration}, which is the unique filtration of the form
\[
\begin{tikzcd}[column sep=0.1cm]
	0  \arrow{rr} && X_{-\ell} \arrow{dl} \arrow{rr} && \cdots \arrow{rr} \arrow{dl} && X_{-k+1} \arrow{rr} \arrow{dl} && X_{-k} = X \arrow{dl}\\
	& H^\ell(X)[-\ell] \arrow[dashed]{ul}{+1} && H^{\ell+1}(X)[-\ell-1] \arrow[dashed]{ul}{+1} && \cdots \arrow[dashed]{ul} &&H^k(X)[-k] \arrow[dashed]{ul}{+1}
\end{tikzcd}
\]  
\end{defprop}


\subsection{The Khovanov--Seidel curve model}
\label{subsec:KSCurveModel}

\afterpage{
\begin{figure}
\centerline{
	\begin{tabular}{c@{\qquad}c@{\qquad}c}
		Spherical objects in bounded 
		&
		&
		\\
		homotopy category of graded projective 
		&
		$\smash{\xrightarrow[\quad\text{combinatorial model}\quad]{}}$
		&
		Curves
		\\
		modules over the zigzag algebra
		&
		&
                \\
		\begin{tikzcd}[ampersand replacement=\&, row sep=-0.1cm, column sep=0.5cm]
			\&           \& P_3 \langle -1 \rangle\arrow{dr} \& \\
			\&           \&                      \& P_2\\
			P_1\langle -5 \rangle\arrow{r}    \& P_1\langle -3 \rangle\arrow{r}   \& P_1 \langle -1 \rangle\arrow{ur} \& 
		\end{tikzcd}
		&
		&
		\raisebox{-.5cm}{\includegraphics[scale=1.3]{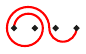}}
		\\[.5cm]
		{\scriptsize cohomology} $\Bigg\downarrow$ {\scriptsize filtration}
		&
		&
		$\Bigg\downarrow$ {\scriptsize cut}
		\\[.5cm]
		Heart of the standard $t$-structure
		&
		$\xrightarrow[\quad\text{combinatorial model}\quad]{}$
		&
		Wiggly arcs
		\\
		\begin{tikzcd}[ampersand replacement=\&, row sep=-0.1cm, column sep=0.5cm]
			P_1\langle -1 \rangle \arrow{dr}\& \\
			\& P_2 \\
			P_3 \langle -1 \rangle \arrow{ur}
		\end{tikzcd}
		&
          	&
		\raisebox{-.5cm}{\includegraphics[scale=1.3]{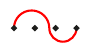}}
		\\
		\(P_1\)
		&
		&
		\raisebox{-.5cm}{\includegraphics[scale=1.3]{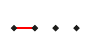}	}
	\end{tabular}
}
\caption{An illustration of the Khovanov--Seidel curve model and \cref{prop:cohomology-pieces}.}
\label{fig:curveModel}
\end{figure}
}

In this section we briefly recall how to recover an object of \(\mathcal{C}_{n+1}\) from a wiggly arc on \(n+2\) points on a line.
In fact, the procedure works on more general curves as well (namely, those that are not necessarily \(x\)-monotone).

Consider \(n+2\) points arranged in a straight horizontal line, and labelled \(0,1, \ldots, n+1\) as in~\cref{sec:combinatorics}.
In addition to the wiggly arcs, consider isotopy classes of curves whose endpoints lie on distinct marked points, and which do not cross themselves or any marked point in their interior.
For convenience, we refer to such an isotopy class as a \defn{curve}.

The Khovanov--Seidel correspondence (explained in Section 4 of~\cite{kho.sei:02}) associates objects of~\(\mathcal{C}_{n+1}\) (up to an overall grading and homological shift) to curves on a configuration of \(n+2\) marked points arranged in a straight line.
We illustrate this correspondence and~\cref{prop:cohomology-pieces} in~\cref{fig:curveModel}.
In particular, the line segment joining point \(i\) to point \(i+1\) is sent to the object \(P_{i+1}\) in the category \(\mathcal{C}_{n+1}\).
Moreover, the minimal intersection number between any pair of curves is the number of morphisms between the associated objects~\cite[Prop.~4.9]{kho.sei:02}.
In fact, the proposition in~\cite{kho.sei:02} is a stronger result: by considering \emph{bigraded curves}, we obtain \emph{bigraded intersection numbers} between the two associated objects.

Any object \(X\) obtained via the Khovanov--Seidel correspondence is~\defn{spherical}; that is, \(\Hom^{g,0}(X,X)\) is \(k\) when \(g = 0,2\), and \(\Hom^{g,h}(X,X) = 0\) in all other cases.
A spherical object \(X\) is necessarily indecomposable as \(\Hom^{0,0}(X,X) = k\).

\begin{proposition}
\label{prop:cohomology-pieces}
Let \(\alpha\) be any curve, and let \(X = X^{\bullet}\) be a chain complex that corresponds to \(\alpha\) under the Khovanov--Seidel reconstruction.
Recall that for any integer \(j\), the cohomology object~\(H^j(X)\) is an object of \(\mathcal{A}_m\), and hence a direct sum of indecomposable objects of \(\mathcal{A}_m\).
\begin{enumerate}
\item The object \(X\) is a (spherical) object of \(\mathcal{A}_m\) up to shift if and only if \(\alpha\) is a wiggly arc. In this case, \(H^j(X) \neq 0\) for exactly one \(j\), such that \(X \cong H^j(X)[-j]\).
\item If \(\alpha\) is a general curve, each indecomposable summand of \(H^j(X)\) for any \(j\) is a spherical object of \(\mathcal{A}_m\) that arises from a wiggly arc, up to shift.
\end{enumerate}
\end{proposition}

\begin{proof}
  Let \(X\) be a complex built from an arbitrary (not necessarily wiggly) curve \(\alpha\).
  Recall from~\cite{kho.sei:02} that \(X\) is spherical, and hence indecomposable.
  Let us now analyse the indecomposable summands of its cohomology objects.
  
  We recall that the procedure to build up \(X\) from \(\alpha\) explicitly describes the differential.
  We analyse the possible non-zero components of the differential, checking how each interacts with the level.
  The only two possibilities are as follows:
  \begin{enumerate}
  \item a map \(P_i\langle g \rangle[h] \to P_{i\pm1}\langle g+1 \rangle[h-1]\), namely multiplication on the right by the path~\((i,i\pm1)\), which does not change the level; or
  \item a map \(P_i\langle g \rangle[h] \to P_{i}\langle g+2 \rangle[h-1]\), namely multiplication on the right by the ``loop'' path~\((i,i\pm1, i)\), which raises the level by one.
  \end{enumerate}
  The first possibility appears each time the curve wiggles either above the \(i\)th point or below the~\((i-1)\)th point.
  The second possibility appears each time the curve makes a U-turn around the right of \(i\)th or the left of the \((i-1)\)th point.

  Computing the \(j\)th cohomology \(H^j(X)\) simply amounts to truncating the complex to level \(j\).
  Explicitly, the truncation is a new chain complex in which we only keep the terms of \(X\) that have level \(j\), together with all maps between them, and forget all the other terms.
  From the calculation above, we see that the resulting complex only has components in its differential of the first type, namely \(P_i\langle g \rangle[h] \to P_{i\pm1}\langle g+1 \rangle[h-1]\) where \(g + h = j\).
  We kill all components of the second type (the ``loop'' components).
  The resulting chain complex breaks up into indecomposable summands, each of which is an object of~\(\mathcal{A}_m[j]\), and this direct sum is \(H^j(X)[j]\).
  
  The object \(X\) is in the heart up to shift if and only if it is concentrated in a single level \(j\).
  Truncating to level \(j\) as described above, we do not delete any of the terms of \(X\).
  This can happen if and only if the differential of \(X\) contains no ``loop'' components,
  which appear in the complex if and only if \(\alpha\) makes a U-turn around one of the \(n+2\) points.
  Thus \(X\) is in the heart up to shift if and only if \(\alpha\) is  \(x\)-monotone.
  This proves the first part of the proposition.

  More generally, consider the indecomposable pieces of the level-\(j\) truncations of \(X\) for various~\(j\), up to shift.
  As explained above, in terms of the curve \(\alpha\), this corresponds to breaking up the curve~\(\alpha\) into \(x\)-monotone portions.
  More precisely, we begin at an endpoint \(p_0\) of \(\alpha\) and follow the curve until it either ends at a point \(p_1\), or stops being \(x\)-monotone by making a U-turn around a point \(p_1\).
  In this case, we take the portion of \(\alpha\) between \(p_0\) and \(p_1\), calling it \(\alpha_1\), and setting its endpoints to be \(p_0\) and \(p_1\) respectively.
  Note that \(p_0 \neq p_1\) because \(\alpha\) is \(x\)-monotone between~\(p_0\) and~\(p_1\).
  We continue with the remainder of \(\alpha\) starting at \(p_1\), greedily taking the maximal \(x\)-monotone portion each time, to obtain pieces \(\alpha_1,\alpha_2, \ldots, \alpha_k\), which are all wiggly arcs.
  These wiggly arcs correspond to complexes \(X_1, \ldots, X_k\), which are spherical objects of \(\mathcal{A}_m\), up to shift.
  It is clear from the construction of the associated complex that these objects, up to shift, are exactly the indecomposable pieces of the cohomology truncations \(H^j(X)\) for various \(j\).
  This argument proves the second part of the proposition.
\end{proof}


\subsection{Cohomology filtrations of spherical objects}
\label{subsec:cohomologyFiltrations}

In this section we explain how wiggly pseudotriangulations govern the combinatorial structure of a natural categorical decomposition of spherical objects of \(\mathcal{C}_m\).

For a general object \(X\) of \(\mathcal{C}_m\), the indecomposable factors of the cohomology filtration can be arbitrary.
However, it turns out that the cohomology filtration is highly constrained if \(X\) is spherical, and these constraints are given by wiggly pseudotriangulations.

\begin{proposition}
\label{prop:hn-filtration}
  Let \(X \in \mathcal{C}_{n+1}\) be a spherical object.
  Consider its cohomology filtration
  \[
    \begin{tikzcd}[column sep=0.3cm]
      0 \arrow{rr} && X_{-\ell} \arrow{dl} \arrow{rr} && \cdots \arrow{rr} \arrow{dl} && X_{-k} = X \arrow{dl}\\
      & H^\ell(X)[-\ell] \arrow[dashed]{ul}{+1} && H^{\ell+1}(X)[-\ell-1] \arrow[dashed]{ul}{+1} & \cdots & H^k(X)[-k] \arrow[dashed]{ul}{+1}
    \end{tikzcd}
  \]
  Consider the set of all possible indecomposable summands of the cohomology objects \(\{H^j(X)\}\) taken without multiplicity.
  Let \(T\) be corresponding set of wiggly arcs, as explained in~\cref{prop:cohomology-pieces}.~Then
  \begin{enumerate}
  \item 
    \(T\) is a subset of a wiggly pseudotriangulation,
  \item \(T\) contains at most one of the two external wiggly arcs.
  \end{enumerate}
\end{proposition}

\begin{proof}
  The proof of this proposition is an easy consequence of analysing the cohomology filtration from the point of view of curves on the disk with marked points.
  Let \(X_1, \ldots, X_k\) be the objects in~\(T\) without repetition, which arise from wiggly arcs \(\alpha_1, \ldots, \alpha_k\) as explained in the proof of~\cref{prop:cohomology-pieces}.
  Let \(\beta = (i,j,A,B)\) and \(\beta' = (i',j', A',B')\) be two distinct wiggly arcs among~\(\{\alpha_1, \ldots,\alpha_k\}\).
  We check that \(\beta\) and \(\beta'\) are compatible in the sense of~\cref{def:compatible}.
  
  Recall that \(\beta\) and \(\beta'\) are two pieces of the curve \(\alpha\).
  Therefore \(\beta\) and \(\beta'\) cannot cross, because~\(\alpha\) does not cross itself.
  Suppose that \(\beta\) and \(\beta'\) are non pointed, and without loss of generality, suppose that \(i' = j\).
  This means that the piece \(\beta\) was cut from \(\alpha\) either because \(i' = j\) is an endpoint of \(\alpha\), or because \(\alpha\) loops around the right of the point \(i'\).
  Moreover, if \(i'\) is an endpoint of \(\alpha\), then the portion of \(\alpha\) closest to \(i'\) lies to the left of \(i'\).
  
  Similarly, the piece \(\beta'\) was cut from \(\alpha\) either because \(i' = j\) is an endpoint of \(\alpha\), or because~\(\alpha\) loops around the left of the point \(j\).
  Moreover, if \(j\) is an endpoint of \(\alpha\), then the portion of \(\alpha\) closest to \(j\) lies to the right of \(j\).

\pagebreak
  We have four possible combinations:
  \begin{enumerate}
  \item \(i' = j\) is an endpoint such that \(\alpha\) travels left starting at \(i'\), and \(i' = j\) is an endpoint such that \(\alpha\) travels right starting at \(j\);
  \item \(i' = j\) is an endpoint such that \(\alpha\) travels left starting at \(i'\), and \(\alpha\) loops around the left of the point \(j\);
  \item \(\alpha\) loops around the right of the point \(i' = j\), and \(i' = j\) is an endpoint such that \(\alpha\) travels right starting at \(j\); 
  \item \(\alpha\) loops around the right of the point \(i' = j\), and \(\alpha\) loops around the left of the point \(j\).
  \end{enumerate}
  It is clear that none of these possibilities can occur within the curve \(\alpha\), and hence \(\beta\) and \(\beta'\) must be pointed.
  Therefore \(T = \{\alpha_1, \ldots, \alpha_k\}\) is a subset of a wiggly pseudotriangulation.

  Now suppose for contradiction that \(T\) contains both external edges.
  Then some portion of \(\alpha\) curves around each of the extreme endpoints \(0\) and \(n+1\).
  Consider the outermost strands of \(\alpha\) that curve around \(0\) and \(n+1\) respectively.
  These must then connect to each other, resulting in~\(\alpha\) having a separate closed component that winds around the outside of the \(n+2\) points.
  This is not possible because \(\alpha\) is a connected curve with two endpoints.
  Therefore \(T\) cannot contain both external edges.
\end{proof}

\begin{remark}
Note that in \cref{def:wigglyComplex}, we adopted the convention that the faces of the wiggly complex contain none of the two external wiggly arcs.
In contrast, the wiggly dissections of \cref{prop:hn-filtration} contain at most one of the two external wiggly arcs.
Combinatorially, this is just the Cartesian product of the wiggly complex with an edge.
\end{remark}


\subsection{Categorical constraints on compatibility of curves}
\label{subsec:categoricalCompatibility}

In this section, we give a categorical criterion for two wiggly arcs to be compatible as in~\cref{def:compatible}.
As such, this section is somewhat dual to the previous section: the criterion in this section determines whether a pair of wiggly arcs can belong together in a single wiggly pseudotriangulation, whereas the previous section generated a wiggly pseudotriangulation given any object of the category.

For two objects \(X\) and \(X'\) of \(\mathcal{C}_m\), let \(\mathbb{E}^\ell(X,X')\) denote the direct sum over all level-\(\ell\) morphisms from \(X\) to \(X'\).
That is,
\[\mathbb{E}^\ell(X,X') = \bigoplus_{g + h = \ell}\Hom^{g,h}(X,X').\]

Recall that two wiggly arcs \(\alpha\) and \(\alpha'\) are said to be compatible (as in~\cref{def:compatible}) if they are non-crossing and pointed.
The incompatibility degree of \(\alpha\) and \(\alpha'\) (as in~\cref{def:incompatibilityDegree}) is \(1\) if \(\alpha\) and \(\alpha'\) are non pointed, and the number of crossings of \(\alpha\) and \(\alpha'\) otherwise.
\begin{proposition}
\label{prop:ext1}
Let \(X\) and \(X'\) be objects in the standard heart \(\mathcal{A}_m\) corresponding to two wiggly arcs \(\alpha\) and \(\alpha'\) respectively.
  Then \(\alpha\) and \(\alpha'\) are compatible if and only if \(\mathbb{E}^1(X, X') = 0\).
  More generally, the incompatibility degree of \(\alpha\) and \(\alpha'\) is the dimension of \(\mathbb{E}^1(X,X')\).
\end{proposition}

\begin{proof}
  The proof follows from a specialisation of the bigraded intersection number calculations that appear in Sections 3 and 4 of~\cite{kho.sei:02}.
  The bigraded intersection number as defined in~\cite[Sect.~3d]{kho.sei:02} between the curves \(\alpha\) and \(\alpha'\) is a polynomial in the variables \(q_1^{\pm}\), \(q_2^{\pm}\).
  It follows from~\cite[Prop.~4.9]{kho.sei:02} that this bigraded intersection number equals the Poincar{\'e} polynomial of the morphisms between \(X\) and \(X'\), namely the expression
  \begin{equation}
    \label{eq:bigraded-homs}
  \sum_{s_1,s_2}\dim(\Hom_{\mathcal{C}_m}(X,X'\langle -s_2 \rangle[s_1]))q_1^{s_1}q_2^{s_2}.    
\end{equation}
However, the grading convention in~\cite{kho.sei:02} is not the same as ours --- our grading is by path length, while the grading in~\cite{kho.sei:02} is one for arrows in one direction and zero for arrows in the opposite direction.
So we have to modify the gradings when we construct the complexes.
In particular, the table in~\cite[Lem.~3.20]{kho.sei:02} that computes the bigraded intersection numbers will change, and so will the table in~\cite[Lem.~4.12]{kho.sei:02} that computes the Poincar{\'e} polynomials of \(\Hom(P_k,-)\).

With our path-length grading, note that a morphism in \(\Hom_{\mathcal{C}_m}(X,X'\langle -s_2 \rangle[s_1])\) has level \(s_1 - s_2\); that is, it takes an object \(X\) of level \(0\) to the object \(X'\langle -s_2 \rangle[s_1]\) of level \(s_1 - s_2\).
  We can thus obtain the dimension of all morphisms of level \(\ell\) by setting \(q_1 = q\) and \(q_2 = q^{-1}\) and taking the coefficient of \(q^{\ell}\) in the expression~\eqref{eq:bigraded-homs}.
    The dimension of \(\mathbb{E}^1(X,X')\), is therefore the coefficient of~\(q\) under this specialisation.

It can be checked that under the path-length grading, the space of morphisms arising from an intersection of curves of ``crossing'' type as well as ``non pointed'' type contains exactly one morphism (up to equivalence and scalar multiplication) of level one; that is, a unique element (up to scalar multiplication) of~\(\mathbb{E}^1\).
Morphisms arising from other intersections do not contribute to~\(\mathbb{E}^1\).
This calculation proves the proposition.
\end{proof}

\begin{remark}
  \cref{prop:ext1} is very similar to the condition of being compatible in a cluster-tilting object of a cluster category (see, \eg~\cite[Cor.~4.3]{bua.mar.rei.ea:06}) or in a $\tau$-tilting object (see, \eg~\cite{AdachiIyamaReiten}).
  The condition in terms of crossings of curves already appears for some specializations (see, \eg~\cite{FominShapiroThurston, PaluPilaudPlamondon-surfaces}).
  However, in \cref{prop:ext1}, it depends on shifting the objects~\(X\) and~\(X'\) so that they both lie in the standard heart, which is an auxiliary step that is not required in cluster-tilting and $\tau$-tilting theory.
\end{remark}

\begin{remark}
  As stated,~\cref{prop:ext1} in combination with~\cref{prop:hn-filtration} is not expected to be true in greater generality.
  For instance, for the zigzag algebra of Dynkin type~\(D_4\), one can find two spherical objects with a non-trivial \(\mathbb{E}^1\) that nevertheless appear together in the cohomology filtration of another spherical object.
  However, we expect to see similar constraints on compatibility, which we hope to turn to in forthcoming work.
\end{remark}

Recall from~\cref{rem:incompatibility-deg} that two wiggly arcs \(\alpha\) and \(\alpha'\) are exchangeable if and only if their incompatibility degree is one.
Recall from~\cref{prop:uerp} the wiggly arcs \(\beta\) and \(\beta'\) associated to \(\alpha\) and \(\alpha'\).
Finally, recall from~\cref{lem:linearDependences} that the \(\b{g}\)-vectors of the wiggly arcs \(\alpha\), \(\alpha'\), \(\beta\), and \(\beta'\) satisfy a linear dependence, depending on their type.
Let us interpret this dependence categorically.

Let \(X\), \(X'\), \(Y\) and~\(Y'\) be spherical objects in the standard heart that correspond to \(\alpha\), \(\alpha'\), \(\beta\) and~\(\beta'\) respectively.
First, note by~\cref{prop:ext1} that \(X\) and \(X'\) are exchangeable if and only if there is a unique non-trivial \(\mathbb{E}^1\) between the corresponding objects of the standard heart.
Next, let us understand~\(Y\) and \(Y'\) categorically.

\begin{proposition}\label{prop:extensions}
For \(\alpha\), \(\alpha'\), \(\beta\), \(\beta'\), and \(X\), \(X'\), \(Y\), \(Y'\) as just discussed,
\begin{itemize}
\item if \(\alpha\) and \(\alpha'\) are not pointed, then for some integers \(g\) and \(g'\), we have extensions
  \[0 \to X' \to Y \to X\langle g \rangle[-g] \to 0 \quad \text{and}\quad 0 \to X \to Y' \to X'\langle g' \rangle[-g'] \to 0;\]
\item if \(\alpha\) and \(\alpha'\) are crossing, then for some integer \(g\), we have an extension
  \[0 \to X' \to (Y \oplus Y') \to X\langle g \rangle[-g]\to 0.\]
\end{itemize}
\end{proposition}

\begin{proof}
  The proof is once again a consequence of analysing the extensions explicitly via the complexes from the Khovanov--Seidel reconstruction.
  We do not reproduce the proof here, as the argument is contained in the proof of~\cite[Thm.~5.5]{tho:06}.
  (See also~\cite[Fig.~4]{tho:06}.)
\end{proof}
Note that if we write down the relations in the Grothendieck group of our category that correspond to the extensions from each case of~\cref{prop:extensions}, then we obtain precisely the linear dependences that appear between the \(\b{g}\)-vectors in~\cref{lem:linearDependences}.
More precisely, we have that
\begin{itemize}
\item if \(\alpha\) and \(\alpha'\) are not pointed, then
  \([Y] = [X] + [X'] = [Y']\),
  from which we see \([X] + [X'] = ([Y] + [Y'])/2\);
\item if \(\alpha\) and \(\alpha'\) are crossing, then
  \([X] + [X'] = [Y] + [Y'].\)
\end{itemize}
This suggests that there should be a way to construct the \(\b{g}\)-vectors categorically in such a way that the \(\b{g}\)-vectors also map to the Grothendieck group of the category.
The \(\b{g}\)-vector of an object~\(X\) corresponding to a wiggly arc \(\alpha\) can, of course, be reconstructed from \(\alpha\).
However, at the moment we do not have a purely categorical construction of the \(\b{g}\)-vector of \(X\) that does not refer to \(\alpha\), and it would be interesting to find this.

\begin{question}
  Suppose we have a collection of spherical objects in the standard heart \(\mathcal{A}_m\) that satisfy the compatibility condition of~\cref{prop:ext1}.
  That is, the \(\mathbb{E}^1\) between any pair of objects in this collection is zero.
  \begin{enumerate}
  \item Can we prove properties about a maximal such collection without reference to curves?
  \item In particular, can we understand the combinatorics of maximal such collections without reference to curves?
  \item Given a maximal compatible collection of objects, can we always find a larger spherical object such that the factors of the cohomology filtration of that object are precisely shifts of the objects in the given collection?
  \end{enumerate}
\end{question}


\section*{Acknowledgements}

This work started during the \emph{Winterbraids XIII: School on low dimensional topology, braid groups and interactions}. We are grateful to the organizers Bérénice Delcroix-Oger, Peter Feller, Vincent Florens, Hoel Queffelec for this fruitful event.
We thank Anand Deopurkar and Anthony M. Licata for many useful conversations and ideas.
Finally, we are grateful to an anonymous referee for many relevant suggestions on the presentation of the paper.


\addtocontents{toc}{ \vspace{.1cm} }
\bibliographystyle{alpha}
\bibliography{wigglyhedra}
\label{sec:biblio}

\end{document}